\documentclass[a4paper, 11pt, reqno]{amsart}

\usepackage{xcolor}

 \definecolor{blue}{HTML}{002395}

\usepackage{amssymb,amsthm,amsmath,mathrsfs}
\usepackage{mathtools}
\usepackage{bbm}
\usepackage{stackengine}
\usepackage{enumitem}
\usepackage{soul}
\usepackage{thm-restate}
\usepackage{subcaption}

\usepackage[foot]{amsaddr}

\usepackage{caption,color,graphicx}
\usepackage[dvipsnames]{xcolor}
\usepackage{soul}

\usepackage{hyperref}
\hypersetup{colorlinks  = true,
            urlcolor    = blue,
            linkcolor   = red,
            citecolor   = blue}

\usepackage{tikz-cd}
\usetikzlibrary{decorations.pathreplacing}
\usepackage{verbatim}

\usepackage{environ}
\mathtoolsset{showonlyrefs,showmanualtags}

\usepackage[centering, top=3cm, bottom=3cm]{geometry}

\usepackage{float} 
\usepackage{caption}
\NewDocumentEnvironment{diagram}{ O{} }{%
  \begin{figure}[#1]
  \captionsetup{type=figure,name=Diagram}%
}{%
  \end{figure}
}
\allowdisplaybreaks

\newcommand{\R}{\mathbb{R}}

\newcommand{\supp}{\mathrm{supp}\,}

\renewcommand{\d}{\mathrm{d}}
\newcommand{\proj}{\mathrm{proj}}

\newcommand{\e}{\varepsilon}

\newcommand{\un}[1]{\underline{#1}}

\newcommand{\vertiii}[1]{{\left\vert\kern-0.25ex\left\vert\kern-0.25ex\left\vert #1 \right\vert\kern-0.25ex\right\vert\kern-0.25ex\right\vert}}

\let\oldtocsection=\tocsection
\let\oldtocsubsection=\tocsubsection
\renewcommand{\tocsection}[2]{\hspace{0em}\oldtocsection{#1}{#2}}
\renewcommand{\tocsubsection}[2]{\hspace{1em}\oldtocsubsection{#1}{#2}}

\theoremstyle{plain}
\newtheorem*{theorem*}{Theorem}

\newtheorem{theorem}{Theorem}[section]
\newtheorem{lemma}[theorem]{Lemma}
\newtheorem{proposition}[theorem]{Proposition}
\newtheorem{corollary}[theorem]{Corollary}

\newtheorem{manualtheorem}{Theorem}
\newenvironment{mytheorem}[1][]{%
    \renewcommand\themanualtheorem{#1}%
    \begin{manualtheorem}
}{\end{manualtheorem}}

\newtheorem{manualcorollary}{Corollary}
\newenvironment{mycorollary}[1][]{%
    \renewcommand\themanualcorollary{#1}%
    \begin{manualcorollary}
}{\end{manualcorollary}}

\newtheorem{manualhyp}{Hypothesis}

\theoremstyle{remark}
\newtheorem{definition}[theorem]{Definition}
\newtheorem{remark}[theorem]{Remark}

\newtheorem{notation}[theorem]{Notation}
  
\newtheorem{manualstep}{\bf Step}
\newenvironment{step}[1][]{%
    \renewcommand\themanualstep{#1}%
    \begin{manualstep}\itshape
}{\end{manualstep}}

\IfPackageLoadedTF{MnSymbol}{\newcommand{\coloneqq}{:=}}{}

\newlist{todolist}{itemize}{2}
\setlist[todolist]{label=$\square$}
\usepackage{pifont}

\author[Matheus M.~Castro]{Matheus M. Castro$^{*}$}
\email{\href{mailto:m.manzatto_de_castro@unsw.edu.au}{m.manzatto\_de\_castro@unsw.edu.au}}

\author[Gary Froyland]{Gary Froyland$^{*}$}
\email{\href{mailto:m.manzatto_de_castro@unsw.edu.au}{g.froyland@unsw.edu.au}}

\address{$^{*}$School of Mathematics and Statistics, University of New South Wales,
Sydney, NSW 2052, Australia}

\usetikzlibrary{arrows.meta,positioning,calc}
\usetikzlibrary{intersections}

\newfloat{diagram}{tbp}{lod}    
\floatname{diagram}{Diagram}    

\usepackage[backend = biber, doi = true, url = false, isbn = false, maxbibnames = 10, giveninits=true, style = numeric]{biblatex}
\AtBeginBibliography{\small}
\usetikzlibrary{bending}

\title[On the cardinality of measures of maximal relative entropy]{On the cardinality of measures of maximal relative entropy for smooth skew products} 
\date{\today}

\begin{document}

\begin{abstract}
    Let $\Omega$ and $M$ be compact smooth manifolds and let $\Theta:\Omega\times M\to\Omega\times M$ be a $\mathcal C^{1+\alpha}$ skew-product diffeomorphism over a transitive Anosov base. 
We show that $\Theta$ has at most countably many ergodic hyperbolic measures of maximal relative entropy. When $\dim M=2$, if $\Theta$ has positive relative topological entropy, then $\Theta$ has at most countably many ergodic measures of maximal relative entropy.

\end{abstract}

\maketitle
\tableofcontents

\newpage
\section{Introduction and main results}

\label{sec:1}

For dynamical systems related by a factor map, the natural analogue of maximal entropy is the concept of maximal relative entropy. 
Let \((X,f)\) and \((Y,g)\) be continuous maps on compact metric spaces and let \(\pi\colon X\to Y\) satisfy \(g\circ\pi=\pi\circ f\).
Write \(\mathcal{M}(f)\) and \(\mathcal{M}(g)\) for the sets of invariant Borel probability measures of $f$ and $g$, respectively.
Given $\nu\in\mathcal{M}(g)$, a measure $\mu\in\mathcal{M}(f)$ is a measure of  maximal $\nu$-relative  entropy if $\pi_*\mu= \mu \circ \pi^{-1}=\nu$ and
$$h_\mu(f)=\sup\bigl\{\,h_\eta(f):\eta\in\mathcal{M}(f),\ \pi_*\eta=\nu\,\bigr\}.$$
Equivalently, if $h_{\nu}(g)<\infty$ the Abramov--Rokhlin formula states that $\mu$ maximises the conditional entropy 
$$h_\eta(f\mid \nu) := h_\eta\left(f\mid\pi^{-1}\mathcal{B}_Y\right)=h_\eta(f)-h_\nu(g)$$ 
among all $\pi$-lifts $\eta$ of $\nu$.
This relative viewpoint is captured by the relativised variational principle of Ledrappier--Walters and in Walters’ compensation function, extending the notion of relative entropy to relative equilibrium states \cite{LedrappierWalters77,Walters86}.

The general picture of measures of maximal relative entropy and relative equilibrium states is well advanced in symbolic dynamics, particularly for factor maps between shifts where $X$ is an irreducible subshift of finite type. Petersen--Quas--Shin proved that for each ergodic base measure $\nu$ on $Y$ there are only finitely many ergodic measures of maximal relative entropy on $X$, and they established explicit bounds on this number \cite{Quas}.  These bounds were improved in \cite{AllahbakhshiQuas12}, introducing the notion of class degree. A subsequent work extended these advances to relative equilibrium states in \cite{AllahbakhshiAntonioliYoo19}. A structure theorem for infinite-to-one factor codes $ \pi\colon X\to Y $ between irreducible subshifts of finite type was proved in \cite{YooDecomp17}:  that there exists a sofic shift $Y_1$, such that $\pi=\pi_2\circ\pi_1$. This allowed Yoo to show that if $\nu $ is an equilibrium state on $ Y $ for a sufficiently regular potential, then there is a unique lift on $X$ that is a relative equilibrium state \cite{YooDecomp17}.

Measures of maximal relative entropy and relative equilibrium states naturally arise in random dynamical systems. There, they are usually called quenched relative equilibrium states. Consider a probability-preserving base $(\Omega,\theta,\mathbb{P})$ and the skew-product map
$$\Theta:(\omega,x)\in \Omega\times M\mapsto (\theta(\omega),\,T_\omega (x))\in \Omega\times M.$$
Observe that when taking $X = \Omega\times M$, $Y = \Omega$, $f=\Theta,$ $g=\theta$, and $\pi(\omega,x)=\mathrm{proj}_\Omega(\omega,x)=\omega$ we retrieve the setting of the first paragraph under the assumption that $h_\mathbb P(\theta)<\infty.$

Foundational results of Kifer established existence and uniqueness of relative equilibrium states for random uniformly expanding transformations and identified quenched relative pressure with the principal Lyapunov exponent of the random transfer-operator cocycle \cite{Kifer92}.
Bogenschütz developed parallel notions of metric entropy and pressure for random dynamical systems and proved a variational principle, in the framework commonly used nowadays \cite{Bogenschutz92}.
Thermodynamic formalism for random systems has since been developed across a variety of settings: random symbolic dynamics \cite{BogenschutzGundlach1995ETDS,GundlachKifer2000DCDS,DenkerKiferStadlbauer2008DCDS,Stadlbauer2010StochDyn,Stadlbauer2017ETDS}; smooth expanding maps \cite{Kifer92,MSU11}; non-uniformly smooth expanding maps \cite{ArbietoMatheusOliveira03,StadlbauerSuzukiVarandas2021CMP}; and non-uniformly hyperbolic random interval maps (closed and open, with discontinuities) \cite{AtnipFroylandGonzalezTokmanVaienti2021CMP,AtnipFroylandGonzalezTokmanVaienti2023ETDS,AtnipFroylandGonzalezTokmanVaienti2024DM}; see also the references therein.

Although the quenched relative thermodynamic formalism theory for random dynamical systems is extensive, results for systems generated by diffeomorphisms remain scarce outside the SRB setting (for SRB results, see \cite{Ledrappier-Young,BlumenthalYoung2019,DragicevicFroylandGonzalezTokmanVaienti2020TAMS,Alves2023,Liu2024} and the references therein). 
In particular, except in settings where Gibbs-type bounds are available, there is no general method for bounding the number of measures of maximal relative entropy, let alone the number of relative equilibrium states, in these settings. In this paper, we study how many measures of maximal relative entropy a skew product $\Theta(\omega,x)=(\theta(\omega),T_\omega(x))$ over an Anosov base may admit.  This question is part of an ongoing program to quantify the number of measures of maximal entropy across classes of systems.  

In the setting of smooth diffeomorphisms, Newhouse proved that $\mathcal C^\infty$ diffeomorphisms on compact smooth manifolds always admit a measure of maximal entropy \cite{Newhouse}.  Buzzi later conjectured that for surface diffeomorphisms with positive topological entropy, there are at most countably many ergodic measures of maximal entropy \cite[Conjecture 1]{Buzzi1}. Sarig’s groundbreaking symbolic coding for \(\mathcal C^{1+\alpha}\) surface diffeomorphisms confirmed this countability and also settled a conjecture of Katok on periodic-orbit growth \cite{Sarig2013}.  This result was recently sharpened by Buzzi--Crovisier--Sarig, proving that $\mathcal C^\infty$ surface diffeomorphisms with positive entropy have only finitely many ergodic measures of maximal entropy, and exactly one in the transitive case, \cite{BCS}, which answered another of Buzzi's conjectures \cite[Conjecture 2]{Buzzi1}.  

In higher dimensions, Ben~Ovadia extended Sarig’s coding to non-uniformly hyperbolic diffeomorphisms on arbitrary manifolds, yielding countability results for hyperbolic measures of maximal entropy \cite{BenOvadia2018}.  More recently, Buzzi–Crovisier–Sarig established the strong positive recurrence (SPR) property, and via SPR obtained strong statistical properties for the associated invariant measures, including exponential decay of correlations and an almost-sure invariance principle \cite{buzzi2025strongpositiverecurrenceexponential}.  Related results are known for two-dimensional fibre skew products over hyperbolic bases; see Marín–Poletti–Veiga \cite{marin2025exponentialmixingmeasuresmaximal}. Analogous results are available for flows  \cite{Lima,Lima3,lima2025symbolicdynamicsnonuniformlyhyperbolic} and discontinuous maps or maps with singularities \cite{YuriCarlos,Yuri2,Poletti1,lima2024measuresmaximalentropynonuniformly}.

Our contribution is a relative analogue of these countable cardinality results. For $\mathcal{C}^{1+\alpha}$ skew–product diffeomorphisms $\Theta(\omega,x)=(\theta(\omega), T_\omega (x))$ over a transitive Anosov base $(\Omega,\theta)$ and for any ergodic $\theta$-invariant measure $\mathbb{P}$ with full support, we prove that there are only countably many hyperbolic ergodic $\Theta$-invariant measures of maximal $\mathbb P$-relative entropy. In particular, when the fibre \(M\) is a surface and $h_{\mathrm{top}}(\Theta\mid\mathbb P)>0$, the set of ergodic \(\mathbb{P}\)-relative measures of maximal entropy is countable. The proof refines the symbolic coding methods of \cite{Sarig2013,buzzi2025strongpositiverecurrenceexponential} to the skew-product setting and combines them with the finiteness techniques for factor maps due to Petersen--Quas--Shin \cite{Quas}. We expect the coding-and-factor approach developed here to be useful in random dynamics, and to provide a natural quenched perspective and new tools for constructing and counting relative equilibrium states.

\subsection*{Structure of the paper.} The paper is organised as follows. The remainder of Section~\ref{sec:1} states the main results and presents an application. Section~\ref{sec:2} recalls background from Pesin theory and reviews homoclinic classes. In Section~\ref{sec:3} we construct a natural symbolic coding for the skew product $\Theta$ that preserves the base dynamics. Section~\ref{sec:4} generalises the results of \cite{Quas} to countable-state Markov shifts. Finally, Section~\ref{sec:5} contains the proofs of the main results, Theorem~\ref{thm:A} and Corollary~\ref{cor:B} build upon the results established in Sections~\ref{sec:3} and \ref{sec:4}.

\subsection{Measures of maximal relative entropy}
Throughout this paper, $\Omega$ and $M$ will always denote compact smooth (boundaryless) manifolds. Let $\Theta$ be a $\mathcal C^{1+\alpha}$, $\alpha>0$, skew-product diffeomorphism
\begin{align}
    \Theta:\Omega\times M &\to \Omega\times M\label{eq:skew product}\\
    (\omega,x)&\mapsto (\theta \omega, T_\omega(x))\nonumber
\end{align}
where $\theta: \Omega\to \Omega$, $T_\omega : M\to M$ for each $\omega\in\Omega$.  Moreover, we define $$\mathrm{proj}_\Omega:(\omega,x)\in \Omega \times M\mapsto  \omega\in \Omega.$$
In this paper, all measures are implicitly assumed to be Borel measures.

This paper concerns the amount of $\mathbb P$-relative entropy that $\Theta$ may admit. Below, we define the metric entropy via the Brin--Katok formula \cite{BrinKatok1983LocalEntropy} since it is the shortest way to define it (see also \cite[Chapter 4]{WaltersBook}). We also define the concept of $\mathbb P$-relative entropy for skew products via an extension of the Brin--Katok formula (see \cite{Yujun}).

\begin{definition}[Metric entropy]
Let $(X,\mathrm{dist})$ be a compact metric space, $T:X\to X$ continuous, and $\mu$ a $T$-invariant  probability measure. For $x\in X$, $\varepsilon>0$, $n\in\mathbb N$, define the Bowen ball
$$B_n(x,\varepsilon):=\left\{y\in X;\ \mathrm{dist}(T^k (y), T^k(x))<\varepsilon \text{ for } 0\leq k<n\right\}.$$
The \emph{metric entropy of $\mu$ with respect to $T$} is
$$h_\mu(T):= - \int_X \lim_{\e\to 0}\limsup_{n\to\infty} \frac{1}{n}\log\mu(B_n(x,\e))\ \d\mu = - \int_X \lim_{\e\to 0}\liminf_{n\to\infty} \frac{1}{n}\log\mu(B_n(x,\e))\ d\mu.$$
\end{definition}

\begin{definition}[Relative metric entropy]
Let $\Theta$ be as \eqref{eq:skew product} and $\mu$ be a $\Theta$-invariant probability measure with $(\mathrm{proj}_\Omega)_*\mu=\mathbb P$, and disintegrate $\mu(\mathrm{d}\omega,\mathrm{d}x)=\mu_\omega(\mathrm{d}x)\,\mathbb P(\mathrm{d}\omega).$
For $\omega\in\Omega$, $x\in M$, $\varepsilon>0$, $n\in\mathbb N$, write
$$B_n^\omega(x,\varepsilon):=\left\{y\in M;\ \mathrm{dist}\left(T_\omega^{k}(y),T_\omega^{k}(x)\right)<\varepsilon \text{ for } 0\leq k<n\right\},$$
where $T_\omega^k(x) := T_{\theta^{k}\omega}\circ T_{\theta^{k-1}\omega}\circ \cdots \circ T_{\omega}(x).$ The \emph{$\mathbb P$-relative metric entropy of $\mu$ with respect to $\Theta$} is
\begin{align*}
h_\mu(\Theta\mid\mathbb P)&:=-\int_\Omega \int_M \lim_{\e\to 0}\limsup_{n\to \infty}  \frac{1} {n}\log(\mu_\omega(B_n^\omega(x,\e) )) \,\mu_\omega(\mathrm{d}x)\,\mathbb P(\mathrm{d}\omega)\\
&=-\int_\Omega \int_M \lim_{\e\to 0}\liminf_{n\to \infty}  \frac{1}{n}\log(\mu_\omega(B_n^\omega(x,\e) )) \,\mu_\omega(\mathrm{d}x)\,\mathbb P(\mathrm{d}\omega).\\   
\end{align*}
We recall that if $h_\mathbb P(\theta) <\infty$, which will always be the case in this paper, we have that $h_\mu(\Theta\mid \mathbb P) = h_\mu(\Theta) - h_\mathbb P(\theta)$ (see \cite[Chapter 5, below equation 1.1.5]{KatokKifer}).

\end{definition}
\begin{definition}[Measures of maximal $\mathbb P$-relative entropy]
Let $\Theta$ be as \eqref{eq:skew product} and fix a $\theta$-invariant probability measure $\mathbb P$. We define the \emph{$\mathbb P$-relative topological  entropy} as
$$h_\mathrm{top}(\Theta\mid \mathbb P) := \sup\{h_\nu(\Theta\mid\mathbb P);\ (\proj_\Omega)_*\nu = \mathbb P\ \text{and}\ \nu\ \text{is a }\Theta\text{-invariant probability measure}\}.$$
We say that a $\Theta$-invariant measure $\mu$ is a \emph{measure of maximal $\mathbb P$-relative entropy} if
$$h_\mu(\Theta\mid \mathbb P) =h_{\mathrm{top}}(\Theta\mid \mathbb P). $$

\end{definition}

To formulate our results, we impose additional assumptions on the skew-product base $\theta:\Omega\to\Omega$ and the measure $\mathbb P$, namely, we assume $\theta$ is a transitive Anosov diffeomorphism and $\mathbb P$ is a $\theta$-invariant probability measure of full support, i.e.
$$\supp \mathbb P:= \{\omega\in\Omega;\, \mathbb P[U]>0\ \text{for every open neighbourhood $U$ of $\omega$}\}=\Omega$$

\begin{definition}
Let \(\Omega\) be a compact connected smooth manifold and let \(\theta:\Omega\to\Omega\) be a \(\mathcal C^{1+\alpha}\) diffeomorphism.  We say that \(\theta\) is \emph{Anosov} if there exist a continuous \(\mathrm{D}\theta\)-invariant splitting $T_\omega\Omega=E^s(\omega)\oplus E^u(\omega),$ and constants \(K\ge1\), \(\chi>0\) such that for all \(k\ge0\),
$$ \mathrm{max}\left(\left\|\left.\mathrm D \theta^k\right|_{E^s(\omega')}\right\|,\left\|\left.\mathrm D \theta^{-k}\right|_{E^u(\omega')}\right\| \right)\leq K e^{-\chi k}.$$

 We say that \(\theta\) is a \emph{transitive Anosov} diffeomorphism if \(\theta\) is Anosov and \(\theta\) is topologically transitive, i.e. there exists \(\omega\in\Omega\) with \(\overline{\{\theta^k(\omega):k\in\mathbb Z\}}=\Omega\).
\end{definition}

\subsection{Main results}

Below we state our main results. The definition of hyperbolic measure used here is recalled in Definition~\ref{def:LE}.

\begin{mytheorem}[A] \label{thm:A} Let $\Omega$ and $M$ be compact smooth manifolds, and 
\begin{align*}
    \Theta:\Omega\times M &\to \Omega\times M\\
    (\omega,x)&\mapsto (\theta \omega, T_\omega(x))
\end{align*}be a $\mathcal C^{1+\alpha}$ diffeomorphism, for some $\alpha>0$. Assume that $\theta:\Omega\to\Omega$ is a transitive Anosov diffeomorphism and $\mathbb P$ is a $\theta$-invariant ergodic probability measure with full support. Then, $\Theta$ admits at most countably many hyperbolic ergodic measures of maximal $\mathbb{P}$-relative entropy.
\end{mytheorem}

Theorem  \ref{thm:A} is proved in Section \ref{sec:5}.  We mention that many natural measures satisfy the full-support assumption; for example, since $\theta$ is a transitive Anosov diffeomorphism, every ergodic equilibrium state for a H\"older potential has full support (see \cite{BowenGibbs}).

We remark that Theorem \ref{thm:A} is a generalisation of the classical non-skew-product setting. Let $T:M\to M$ be a $\mathcal C^{1+\alpha}$ map, and $\theta:\mathbb T^2\to\mathbb T^2$ be Arnold’s cat map
$$\theta\left(\omega_{1},\omega_{2}\right)
=\left(2\omega_{1}+\omega_{2},\,\omega_{1}+\omega_{2}\right)\ \mathrm{mod}\ 1.$$ Consider the product map $\Theta(\omega,x)=(\theta(\omega),T(x))$ and let \(\mathbb P\) be Lebesgue measure on \(\mathbb T^{2}\) (which is \(\theta\)-ergodic). Then any $\Theta$-invariant measure
$\mu$ satisfying $(\mathrm{proj}_\Omega)_*\mu=\mathbb P$ has the form $\mu=\mathbb P\otimes\mu'$, where $\mu'$ is $T$-invariant on $M$, and
$$h_{\mu}\big(\Theta\,\big|\,\mathbb P\big)=h_{\mu'}(T).$$
Hence, measures of maximal $\mathbb P$-relative entropy for $\Theta$
correspond exactly to measures of maximal entropy for $T$. In particular,
Theorem~\ref{thm:A} yields the usual countability statement for ergodic
hyperbolic measures of maximal entropy in the non–skew-product setting in conformance with \cite[Theorem 0.5]{BenOvadia2018}.

We stress that the hyperbolicity assumption in Theorem~\ref{thm:A} is essential. Observe that the extension of $\theta$ by the identity map $\Theta(\omega,x) = (\theta(\omega),x)$  admits uncountably many ergodic measures of maximal $\mathbb P$-relative entropy; in this case, all such measures have zero $\mathbb P$-relative entropy. A similar example with positive $\mathbb P$-relative entropy is
$\Theta:\Omega\times \mathbb T^3\to \Omega\times \mathbb T^3$ with 
$$\Theta(\omega,(x_1,x_2,y)) = (\theta \omega, (2 x_1 + x_2, x_1 + x_2,y)),$$
which admits uncountably many ergodic measures of maximal $\mathbb P$-relative entropy.

Moreover, for fibre maps in dimension $n\geq 3$, the countability conclusion for hyperbolic measures of maximal entropy in Theorem \ref{thm:A} is sharp, i.e.\ there exists a $\mathcal C^{\infty}$ skew-product diffeomorphism $\Theta:\Omega\times\mathbb T^n\to\Omega\times\mathbb T^n$ satisfying the hypotheses of Theorem \ref{thm:A} which admits infinitely many hyperbolic ergodic measures of maximal $\mathbb P$-relative entropy. As before, take $\mathbb P$ as the Lebesgue measure in $\mathbb T^2$ and $\theta:\mathbb T^2\to\mathbb T^2$ as Arnold's cat map, and set $\Theta(\omega,x)=(\theta(\omega),T(x))$, where $T:\mathbb T^n\to\mathbb T^n$ admits infinitely many hyperbolic measures of maximal entropy. One possible construction of $T$ is as follows.

Define
$$
T:(x,s)\in \mathbb T^{n-1}\times \mathbb S^1 \mapsto (A(x),F(s))\in \mathbb T^{n-1}\times \mathbb S^1,
$$
where $A\in \mathrm{SL}(n-1,\mathbb Z)$ is a transitive Anosov toral automorphism, and $F\colon \mathbb S^1\to \mathbb S^1$ is the function
$$F(s)= s + \delta\, e^{-\left(\frac{1}{s^2} + \frac{1}{(1-s)^2}\right)}  \sin\left(\frac{2\pi}{s}\right)\pmod{1},
$$
where $\delta>0$ is taken sufficiently small so that $F$ is a diffeomorphism and $F'(s)>0$ for every $s\in\mathbb S^1$. One readily verifies that $F\in\mathcal C^\infty$ and 
$$F\left(\frac{1}{j}\right)= \frac{1}{j}\ \text{and }0<F'\left(\frac{1}{j}\right) = 1  -2\delta \pi j^2 \,   e^{-\left(1+\frac{1}{(j-1)^2}\right) j^2} <1\ \text{for every }j\in \mathbb N_{>1} .$$
 Since circle diffeomorphisms satisfy $h_{\mathrm{top}}(F)=0$, we obtain that $h_{\mathrm{top}}(T)=h_{\mathrm{top}}(A)$. If $m$ is the (unique) measure of maximal entropy for $A$ (in fact, the Lebesgue measure on $\mathbb T^{n-1}$), then for each $j\in\mathbb N_{>1}$ the product $
m_j \coloneqq m\otimes \delta_{1/j}$ is ergodic, hyperbolic, and satisfies $h_{m_j}(T)=h_{m}(A)=h_{\mathrm{top}}(T)$. Thus $T$ admits infinitely many hyperbolic ergodic measures of maximal entropy.

In dimension two, Theorem \ref{thm:A} can be strengthened.

\begin{mycorollary}[B]\label{cor:B} Under the assumptions of Theorem \ref{thm:A} if we assume that $M$ is a compact smooth surface and $h_{\mathrm{top}}(\Theta\mid\mathbb P)>0$  then $\Theta$ admits at most countably many measures of maximal $\mathbb P$-relative entropy.
\end{mycorollary}
Corollary \ref{cor:B} is proved in Section \ref{sec:5}.

In dimension two, it may be possible to strengthen the conclusion of ``at most countably many measures of maximal $\mathbb P$-relative entropy'' of Corollary \ref{cor:B} to ``at most finitely many measures of maximal $\mathbb P$-relative entropy'', in analogy with the classical non-skew-product case \cite[Corollary 1.4]{BCS}.
Another natural future direction is to derive a criterion for uniqueness. In \cite{BCS} (see also \cite{buzzi2025strongpositiverecurrenceexponential}) it is shown that each homoclinic class supports a unique measure of maximal entropy. In particular, any smooth, transitive surface system with positive entropy has a unique measure of maximal entropy. The analogous statement for measures of maximal relative entropy remains unclear. Finally, we remark that the  Anosov hypothesis may extend to Axiom A, but the symbolic system for $\Theta$ used in Theorem~\ref{thm:1block} would require a different construction.

It is worth mentioning that, as a consequence of a celebrated result due to Newhouse \cite{Newhouse}, $\mathcal C^\infty$ skew products always admit a measure of maximal $\mathbb P$-relative entropy.
\begin{proposition} \label{prop:newhouse}Assume that the skew product $\Theta$ is $\mathcal C^\infty,$ and let $\mathbb P$ be any $\theta$-invariant measure. Then $\Theta$ admits a  $\mathbb P$-relative measure of maximal entropy.
\end{proposition}
\begin{proof} From \cite[Theorem 4.1]{Newhouse} the map 
$$\mu \in \left\{\nu;\ \nu\text{ is a $\Theta$-invariant probability measure}\right\}\mapsto h_\mu (\Theta)\in \mathbb R$$
 is upper semi-continuous with respect to the weak$^*$ topology. Since $$ \mathcal M_{\mathbb P}(\Theta):=\left\{\nu;\ \nu\text{ is }\Theta\text{-invariant probability measure } \text{ and }\nu(\mathrm{proj}_\Omega^{-1}(\cdot)) = \mathbb P\right\}$$
 is compact in the weak$^*$ topology, there exists  $\widetilde{\mu}\in \mathcal M_\mathbb P (\Theta)$ such that
 $$h_{\widetilde{\mu}}(\Theta) = \sup\left\{h_\nu(\Theta);\ \nu\in\mathcal M_{\mathbb P}(\Theta)\right\}.$$
 Since $h_{\nu}(\Theta\mid \mathbb P) = h_{\nu}(\Theta) - h_{\mathbb P}(\theta)$ for each $\nu\in \mathcal M_\mathbb P(\Theta)$ we obtain that $\widetilde{\mu}$ is a measure of maximal $\mathbb P$-relative entropy.
\end{proof}

Throughout the paper, if $X$ is a topological space and $T:X\to X$ is a measurable map, we denote:
\begin{itemize}
\item $\mathcal B\left(X\right)$ as the Borel $\sigma$-algebra of $X$;
    \item $\mathcal M(X):=\left\{\mu;\ \mu \ \text{is a positive measure on }X\right\};$ 
    \item $\mathcal M(T):=\left\{\mu\in \mathcal M\left(X\right);\ T_*\mu := \mu(T^{-1}(\cdot)) = \mu\ \text{and }\mu\ \text{is a probability measure}\right\};$
    \item $\mathcal M_e(T):= \{\mu \in \mathcal M (T);\ \mu\ \text{is ergodic}\}$.
\end{itemize}

\subsection{An application of Corollary \ref{cor:B}} 
Let \(\mathbb T^2=\mathbb R^2/\mathbb Z^2\) be the flat torus. A classical and notoriously challenging dynamical system is the standard map on $\mathbb T^2$. One convenient version of the standard map is \(T:\mathbb T^2\to\mathbb T^2\) given by
\[
T(x_1,x_2)=(2x_1-x_2+k\cos(2\pi x_1),\,x_1)\pmod{1},\ \text{for some } k>0.
\]
Sinai conjectured that for $k$ sufficiently large, $T$ has positive metric entropy with respect to Lebesgue measure \cite[Page 144]{Sinai1994}. This conjecture remains open. In contrast, much more is known about measures of maximal entropy: Obata showed that, for large $k$, the standard map admits a unique measure of maximal entropy; moreover, such a measure has positive metric entropy \cite{Obata2021}. However, less is known in the random setting. 
Our results show that there exists at least one and, at most, countably many ergodic relative measures of maximal entropy for certain classes of random perturbations of the standard map.

For each $k>0$, we consider the skew product 
\begin{align*}
    \Theta_k:\mathbb T^2 \times \mathbb T^2 &\to \mathbb T^2 \times \mathbb T^2\\
    ( \omega,x)&\mapsto (A \omega, T_\omega(x)), 
\end{align*}
where $A:\mathbb T^2\to \mathbb T^2$ is Arnold's cat map $$A(\omega_1,\omega_2) = (2\omega_1 + \omega_2,\omega_1+\omega_2)\pmod{1},$$
and $T_\omega:\mathbb T^2\to\mathbb T^2$ is a perturbed version of the standard map
$$T_\omega(x_1,x_2) := (2 x_1 - x_2+ k \cos(2\pi x_1) + f(\omega) , x_1)\pmod{1},\ \text{for every }\omega\in \Omega,$$
where $f:\mathbb T^2\to\mathbb S^1$ is a $\mathcal C^\infty$ function. 
\begin{proposition}\label{prop:example}
Let $\mathbb P$ be a $\theta$-invariant ergodic measure with full support. Then, there exists $k_0>0$, such that for every $k> k_0,$ one has $h_{\mathrm{top}}(\Theta_k\mid \mathbb{P})>0$. In particular, Proposition \ref{prop:newhouse} implies that $\Theta_k$ admits a measure of maximal $\mathbb P$-relative entropy and Corollary \ref{cor:B} ensures that there are only countably many such ergodic measures. 
\end{proposition}

\begin{remark}
Examples of choices for $f$ are: 
\begin{itemize}
    \item $f(\omega_1,\omega_2) = \omega_1$ (large random additive kicks);
    \item $f(\omega_1,\omega_2) = \e \sin(2\pi \omega_1)$ for any $\e>0$ (arbitrarily small random additive kicks); and
    \item $f(\omega_1,\omega_2)= 2\omega_1 + \e \sin(2\pi \omega_2)$ (random kicks of mixed sizes depending on both base coordinates).
\end{itemize}
\end{remark}

\section{Hyperbolic measures and Borel homoclinic classes}
\label{sec:2}
In this section, we collect standard facts from hyperbolic dynamics and introduce the concept of Borel homoclinic classes (see \cite[Definition 2.7]{buzzi2025strongpositiverecurrenceexponential}). Fix a compact smooth manifold $E$ and a $\mathcal{C}^{1+\alpha}$ diffeomorphism $\Theta:E\to E$, for some $\alpha>0$. Unless stated otherwise, all statements in this section are made under this standing assumption. For the purposes of this paper, the set $E$ can be thought of as $\Omega\times M$ as in Theorem \ref{thm:A}. Our exposition closely follows \cite[Section 2]{buzzi2025strongpositiverecurrenceexponential}.

\subsection{Hyperbolic measures}\label{sec:hyperbolic}
By the Oseledets' multiplicative ergodic theorem \cite{Oseledets68}, there exists a Borel set $E'\subset E$ such that
\begin{itemize}
  \item $E'$ has full measure for every $\Theta$-invariant probability measure; and
  \item for every $x\in E'$ there is a finite set $\sigma(x)\subset\mathbb R$ and a $\mathrm{D} \Theta$-invariant splitting
  $$
  T_xE = \bigoplus_{\lambda\in\sigma(x)} E_x^\lambda,
  $$
  where each $E_x^\lambda\neq\{0\}$ is a linear subspace defined by 
  $$
  E_x^\lambda = \left\{v\in T_xE\setminus\{0\};\ \lim_{n\to\infty}\frac{1}{n}\log\big\|\mathrm D_x \Theta^{n} v\big\|=\lambda\right\}\ \cup\ \{0\}.
  $$
\end{itemize}

\begin{definition}\label{def:LE}
Let $\Theta:E\to E$ be as above and $x\in E'$. The set $$\sigma(x)=\{\lambda_1(x),\ldots,\lambda_{k(x)}(x)\}$$ is called the \emph{Lyapunov spectrum of $\Theta$ at $x$}, its elements are the \emph{Lyapunov exponents of $\Theta$ at $x$}. For $\chi>0$ and $\mu\in\mathcal M(\Theta)$ we say that $\mu$ is \emph{$\chi$-hyperbolic} if, for $\mu$-a.e.\ $x\in E$. 
\begin{itemize}
  \item[(i)] $\min\{|\lambda|;\ \lambda\in\sigma(x)\}\geq \chi$;
  \item[(ii)] there are both positive and negative exponents, i.e.
  $$
  \sigma^{+}(x):=\{\lambda\in\sigma(x):\lambda>0\}\neq \ \varnothing
  \quad\text{and}\ 
  \sigma^{-}(x):=\{\lambda\in\sigma(x):\lambda<0\}\neq\varnothing.
  $$
\end{itemize}
We then set
$$
E_x^{s}:=\bigoplus_{\lambda\in\sigma^{-}(x)} E_x^\lambda
\ \text{and}\ 
E_x^{u}:=\bigoplus_{\lambda\in\sigma^{+}(x)} E_x^\lambda.
$$

If $\mu\in\mathcal M_e(\Theta)$ is ergodic, then there exist real numbers $\lambda_1<\cdots<\lambda_k$ and integers $m_1,\ldots,m_k\geq 1$ such that, for $\mu$-a.e.\ $x\in E$,
$$
\sigma(x)=\{\lambda_1,\ldots,\lambda_k\}
\ \text{and}\ 
\dim E_x^{\lambda_i}=m_i, \text{ for } i=1,\ldots,k.
$$

Finally, we recall that if \textnormal{(i)} holds but \textnormal{(ii)} fails, then $\mu$ is supported on a periodic orbit and hence has zero metric entropy \cite[Corollary S.5.2]{KatokBook} (see also the proof of \cite[Theorem 4.2]{Katok80}).
\end{definition}

Below, we recall a few results of Pesin theory. These definitions will be useful for defining the concept of Borel homoclinic classes in the following section.

\begin{definition}[$(\chi,\e)$-Pesin block] Given $\chi,\e>0$, a \emph{$(\chi,\e)$-Pesin block} is a nonempty set $\Lambda\subset E$ for which there exists a direct sum $T_x E= E^s(x)\oplus E^u(x)$ for all $x\in \bigcup_{n\in\mathbb Z} \Theta^n(\Lambda),$  and a number $K>0$ such that for any $n\in\mathbb Z$, $k\geq 0,$ and $y\in \Lambda$,
$$ \mathrm{max}\left(\left\|\left.\mathrm D \Theta^k\right|_{E^s(\Theta^n(y))}\right\|,\left\|\left.\mathrm D \Theta^{-k}\right|_{E^u(\Theta^n(y))}\right\| \right)\leq K e^{-\chi k + \e |n|}.$$
\end{definition}
We recall that Pesin blocks are not necessarily $\Theta$-invariant sets.

Below we state Pesin's stable manifold theorem, whose proof can be found at \cite{Pesin} (see also \cite[Theorem 2.2]{buzzi2025strongpositiverecurrenceexponential}).
\begin{theorem}\label{thm:localstableunstable} There exists a continuous function $\e_{0}$ with the following property: Given $\chi,\e >0$, with $\e<\e_{0}(\chi)$, and a  $(\chi,\e)$-Pesin block $\Lambda$. Then each $x\in \Lambda$ admits a $\mathcal C^{1+\alpha}$ embedded disc $W_{\mathrm{loc}}^s(x)$ satisfying:
$$\text{for each }y\in W^s_{\mathrm{loc}}(x),\ \limsup_{n\to\infty}\frac{1}{n} \log \mathrm{dist}_{N}(\Theta^n(x),\Theta^n(y))<0. $$
Moreover, the embedded disc $W_{\mathrm{loc}}^s$ varies continuously with $x\in\Lambda$ in the $\mathcal C^{1+\alpha}$ topology. The same result applied to $\Theta^{-1}$ gives the local unstable manifold, which we call $W^u_{\mathrm{loc}}(x).$
\end{theorem}

The above theorem allows us to define the \emph{global stable and unstable manifolds} of $x$ lying in a $(\chi,\e)$-Pesin block, where $(\chi,\e)$ satisfy the assumptions of Theorem \ref{thm:localstableunstable}, as
$$W^s(x) := \bigcup_{n\geq 0} \Theta^{-n}(W^s_{\mathrm{loc}}(\Theta^n(x))) \ \text{and }W^u(x) = \bigcup_{n\geq 0} \Theta^n(W^s_{\mathrm{loc}}(\Theta^{-n}(x))). $$
We recall that the choice of $\e>0$ in Theorem \ref{thm:localstableunstable} does not affect the definition of the global stable and unstable manifolds \cite[§8.2]{PesinBook} (see also \cite[Section 2]{buzzi2025strongpositiverecurrenceexponential}). 

\begin{definition}
    We define the \emph{recurrent non-uniformly hyperbolic set of $\Theta$} as
    $$\mathrm{NUH'}(\Theta) := \bigcup_{\chi>0} \bigcap_{\e>0} \bigcup_{\substack{(\chi,\e)\text{-Pesin}\\\text{ block }\Lambda}} \left\{ x\in E:\, x\ \text{is a limit of periodic points in }\Lambda\right\}. $$
Observe that
$$\mathrm{NUH'}(\Theta)\subset \bigcup_{\chi>0} \bigcap_{\e>0}\{x\in E:\, x\ \text{is contained in a }(\chi,\e)\text{-Pesin block}\}.$$
In particular, from Theorem \ref{thm:localstableunstable} each $x\in \mathrm{NUH}'(\Theta)$ admits global stable and unstable manifolds.
\end{definition}
\begin{theorem}[{\cite[Theorem 4.1]{Katok80}}] Let $\mu\in\mathcal M_e(\Theta)$, then $\mu(\mathrm{NUH}'(\Theta))=1$. \label{thm:katok}
\end{theorem}

\subsection{Borel homoclinic classes} Given two submanifolds $S_1,S_2$ in $E$ we denote $S_1 \pitchfork S_2$ to be the set of transverse intersection points, i.e.
$$S_1\pitchfork S_2 := \{x\in S_1\cap S_2: T_x S_1 \oplus T_x S_2 = T_x E\}. $$

\begin{definition}[\cite{buzzi2025strongpositiverecurrenceexponential}, Homoclinic relation] Two points $x,y\in \mathrm{NUH}'(\Theta)$ are \emph{homoclinically related} $(x\sim y)$ if $W^s(x)$ has a transverse intersection point with an iterate of $W^u(y)$, and $W^u(x)$ has a transverse intersection with an iterate of $W^s(y)$.
\end{definition}

\begin{proposition}[{\cite[Proposition 2.6]{buzzi2025strongpositiverecurrenceexponential}}]
 The homoclinic relation $\sim$ is an equivalence relation on $\mathrm{NUH}'(\Theta)$.  
\end{proposition}

 \begin{definition}[\cite{buzzi2025strongpositiverecurrenceexponential}] \emph{Borel homoclinic classes} are equivalence classes of $\sim$ in $\mathrm{NUH}'(\Theta)$.
 \end{definition}

The proposition below states that the homoclinic classes characterise all the hyperbolic measures.
 \begin{proposition}[{\cite[Proposition 2.8]{buzzi2025strongpositiverecurrenceexponential}}]\label{prop:homoclinicclasses}
  The following hold:
  \begin{enumerate}
      \item[(i)] The set of Borel homoclinic classes is a finite or countable partition of $\mathrm{NUH}'(\Theta)$ into invariant Borel sets;
      \item[(ii)] each class contains a dense set of hyperbolic periodic points;
      \item[(iii)] any hyperbolic ergodic measure $\mu$ is carried by a Borel homoclinic class $\mathcal{X}$; and
      \item[(iv)] any invariant measure carried by a Borel homoclinic class is hyperbolic.
  \end{enumerate}
 \end{proposition}

\begin{remark}
 Note that any Borel homoclinic class that is not a singleton supports uncountably many invariant ergodic measures. If $\Theta$ admits a non-atomic hyperbolic invariant measure (i.e., not supported on a periodic orbit), then $\Theta$ admits uncountably many hyperbolic ergodic measures. Indeed, every non-trivial Borel homoclinic class contains at least one horseshoe, and each horseshoe carries uncountably many ergodic measures since it is conjugated to a full shift. By Proposition \ref{prop:homoclinicclasses} (iv), each of these measures is hyperbolic.
\end{remark}

We conclude this section by introducing the notions of rectangle with local product structure and the Smale bracket, which will be used throughout.
\begin{definition}[Rectangle and Smale bracket]\label{def:rec}
Let $\Theta:E\to E$ be a dynamical system on a smooth manifold $E$. A set $R\subset E$ is a \emph{rectangle with local product structure} if there exist families $\Gamma^s$, $\Gamma^u$ of $\mathcal C^1$ discs embedded in $E$ such that:
\begin{itemize}
\item $R=\bigcup_{\gamma^s\in \Gamma^s}\gamma^s\cap \bigcup_{\gamma^u \in \Gamma^u}\gamma^u$;
\item given $\gamma,\gamma'\in \Gamma^{s}$ (or $\Gamma^u$), then either $\gamma = \gamma'$ or $\gamma \cap \gamma' = \emptyset.$
\item each $\gamma^s \in \Gamma^s$ is a piece of stable manifold and for each $x\in R$, there exists a unique $\gamma^s_x\in \Gamma^s$ such that $x\in \gamma^s_x$;
\item each $\gamma^u \in \Gamma^u$ is a piece of unstable manifold and for each $y\in R$, there exists a unique $\gamma^u_y\in \Gamma^u$ such that $y\in \gamma^u_y$
\item for every $\gamma^s\in\Gamma^s$ and $\gamma^u\in\Gamma^u$ one has $\#(\gamma^s\cap\gamma^u)=1$.
\end{itemize}

Given a rectangle $R$ and $x,y\in R$, define the \emph{Smale bracket}
$$
[x,y]_\Theta:=\gamma^s_x\cap\gamma^u_y,
$$
which is a single point of $R$ by the last item. We also define the sets
$$W^{s}(x,R) := \gamma^s_x\cap R\ \text{and } W^{u}(x,R) := \gamma^u_x\cap R.$$
\end{definition}

\begin{figure}[h!]
    \centering
    \includegraphics[width=0.4\linewidth]{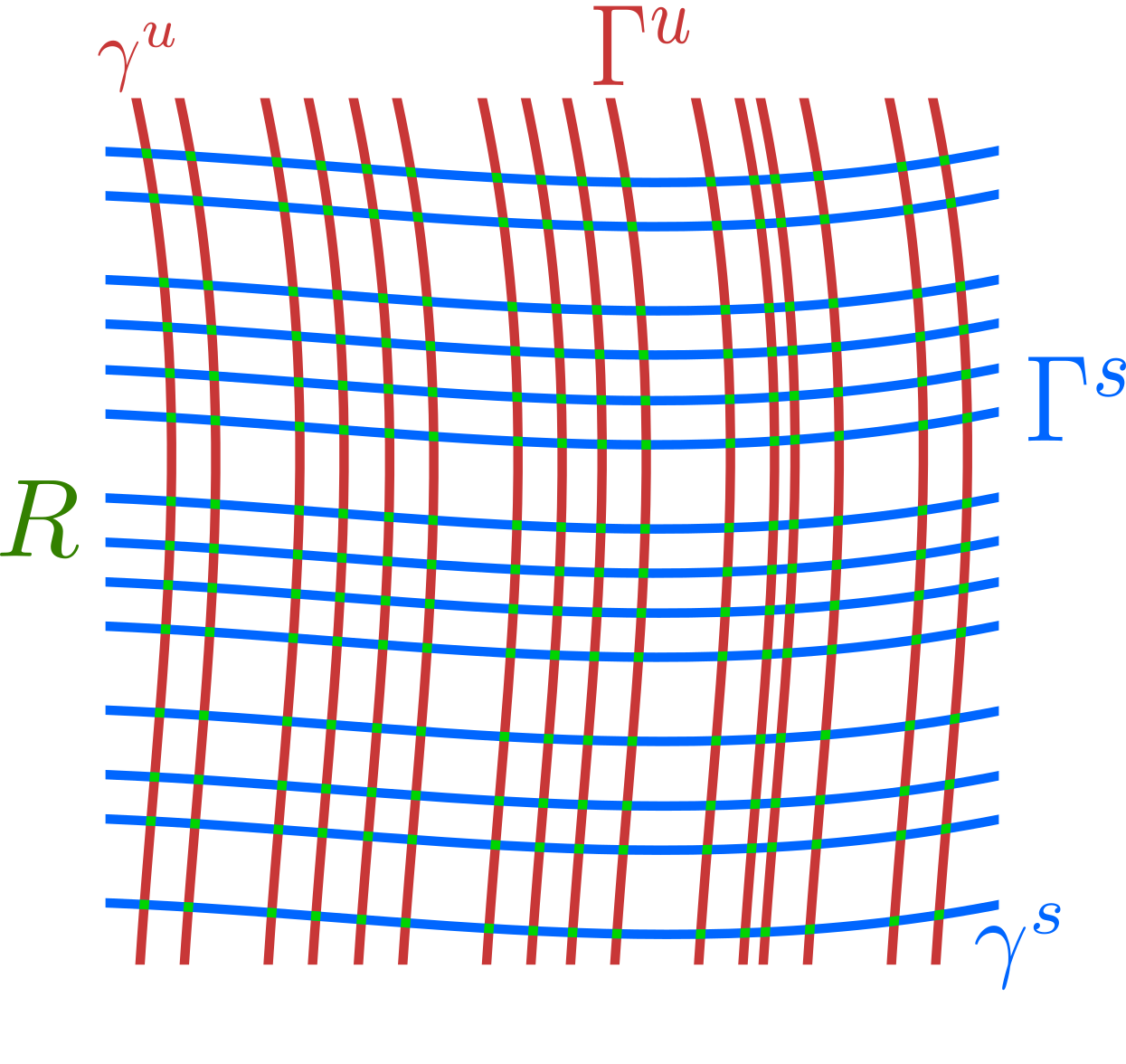}
    \caption{Illustration of a rectangle $R$ (shown as green dots) as the intersection of segments of stable and unstable manifolds (blue and red curves, respectively).}
    \label{fig:placeholder1}
\end{figure}

\section{Topological Markov shifts}
\label{sec:3}

This section proves Theorem \ref{thm:1block}, which enables us to analyse the skew product $\Theta$ via a countable-state topological Markov shift and a factor map that preserves the base dynamics.

\begin{definition}\label{def:tms}
Let $S$ be a countable alphabet and a directed graph $\mathscr{G} = (\mathscr V,\mathscr E)$, such that the set of vertices $ \mathscr V=S$. The \emph{topological Markov shift associated to $\mathscr G$} is the discrete-time topological dynamical system $\sigma :\Sigma\to\Sigma$ where 
$$\Sigma = \Sigma(\mathscr{G}) = \{v=(v_i)_{i\in\mathbb Z}\in S^{\mathbb Z};\ \text{if }v_i\to v_{i+1}\in \mathscr E\ \text{for every }i\in\mathbb Z\},$$
equipped with the metric $d(v,u) = \{\exp(-\min\{|n|; v_n\neq  u_n\})$, and $\sigma:\Sigma \to \Sigma$ is the left shift map $\sigma( (v_i)_{i\in\mathbb Z}) = (v_{i+1})_{i\in\mathbb Z}$. Given a topological Markov shift $\Sigma$ associated to a graph $\mathscr{G}=(\mathscr V,\mathscr E)$ we define the set
\begin{equation}
    \label{Sigmahash}
\Sigma^\# := \left\{(v_i)_{i\in\mathbb Z}\in\Sigma;\, \exists u',v'\in \mathscr V\ \text{s.t. }\,\#\{i\in\mathbb Z_+; v_i = v'\} =\#\{i\in\mathbb Z_{-}; v_i = u'\}=\infty \right\}.
\end{equation}

We say that $\Sigma$ is irreducible if for every two symbols $v_i,v_j \in \mathscr V$, there exists a path in the graph $\mathscr G$ which connects $v_i$ to $v_j$.

\end{definition}

In the following, we recall the definition of a Markov partition.

\begin{definition}[Markov partition]\label{def:BowenPartition}
Let \(\theta:\Omega\to \Omega\) be a $\mathcal C^{1+\alpha}$ dynamical system. A finite family \(\mathscr P=\{P_1,\dots,P_k\}\) of  rectangles with local product structure (see Definition \ref{def:rec}) is a \emph{Markov partition} if:

\begin{itemize}
\item[(1)] for each $i\in\{1,\ldots,k\}$
$$P_i = \overline{\operatorname{int} (P_i)}\ \text{and}\ \operatorname{int} (P_i) \cap  \operatorname{int}P_j = \varnothing$$ if $i\neq j$, where the interior and closure are being taken in the induced topology of $NW(\theta)$, where $NW(\theta)$ denotes the non-wandering set of $\theta:\Omega\to \Omega$;    
  \item[(2)] 
  whenever $\omega\in \operatorname{int}P_i$ and $\theta(\omega)\in \operatorname{int} P_j$,
  $$ \theta\big(W^s(\omega,P_i)\big)\subset W^s(\theta(\omega),P_j)\text{ and }
    \theta\big(W^u(\omega,P_i)\big)\supset W^u(\theta (\omega),P_j).$$

  \item[(3)] define the directed graph \(\mathscr G_b=(\mathscr V_b,\mathscr E_b)\) with
  \(\mathscr V_b=\mathscr P\) and edges
  \[
    \mathscr E_b=\big\{P_i\to P_j;\ \operatorname{int}P_i\cap \theta^{-1}(\operatorname{int}P_j)\neq\varnothing\,\big\},
  \]
  and let \(\Sigma_b=\Sigma_b(\mathscr G_b)\) be the associated subshift of finite type.
  The coding map
  \begin{align}
      \pi_b:(P_{i_n})_{n\in\mathbb Z}\ \mapsto\ \bigcap_{n\in\mathbb Z}\theta^{-n}\big(P_{i_n}\big) \label{def:pib} 
  \end{align}
  is well defined, Hölder continuous, and satisfies the semiconjugacy
  \(\pi_b\circ \sigma=\theta\circ \pi_b\).

  \item[(4)] the map $\pi_b$ is surjective onto $NW(\theta)$ and there exists \(k_{\mathscr P}\in\mathbb N\)
  such that \(\#\,\pi_b^{-1}(\omega)\le k_{\mathscr P}\) for every \(\omega\in NW(\theta)\).
\end{itemize}

\end{definition}

\begin{theorem}[{\cite{Bowen70}, \cite[Appendix~III]{ParryPollicott1990}}]
Every $\mathcal C^{1+\alpha}$ Axiom~A diffeomorphism $\theta:\Omega\to\Omega$ admits a finite Markov partition; in particular, so does every Anosov diffeomorphism.
\end{theorem}

\begin{definition}\label{def:ba}
Let $\theta:\Omega\to\Omega$ be an Anosov diffeomorphism and let $\mathbb P$ be  $\theta$-invariant ergodic probability measure on $\Omega$.
Given a Markov partition $\mathscr P=\{P_1,\ldots,P_\kappa\}$ for $\theta$ we denote
$$
\partial \mathscr P :=\bigcup_{P\in\mathscr P}\partial P.
$$

It is convenient to decompose the boundary into its stable and unstable parts:
$$\partial^s\mathscr P:=\bigcup_{P\in\mathscr P}\partial^s P,\ 
\partial^u\mathscr P:=\bigcup_{P\in\mathscr P}\partial^u P,$$
where for each \(P\in\mathscr P\),
\begin{itemize}
    \item $\partial^s P:=\{x\in P: x\notin \operatorname{int}_{W^u(x)}(W^u(x,P))\}$; and
    \item $\partial^u P:=\{x\in P: x\notin \operatorname{int}_{W^s(x)}(W^s(x,P))\},$
\end{itemize}
where interiors taken in the relative topologies of $W^{s}_{\mathrm{loc}}(x)$ and $W^{u}_{\mathrm{loc}}(x)$, respectively (see \cite[Lemma 3.11]{BowenBook}). Observe that $\mathbb P$ is ergodic, we have that $\mathbb P(\partial \mathscr P) \in \{0,1\}$  (see \cite[Proposition 3.5]{BowenBook}).

\end{definition}

The next statement, taken from \cite[Theorem 10.5 and Section 10]{buzzi2025strongpositiverecurrenceexponential}, is a higher-dimensional generalisation of Sarig’s celebrated theorem \cite[Theorem 1.3]{Sarig2013} (see also \cite[Theorem 0.1]{BenOvadia2018}). We mention that the original result presented in \cite{buzzi2025strongpositiverecurrenceexponential} is more general; however, the presented version is enough for our purposes.

\begin{theorem}[{\cite[Theorem 10.5 and Section 10]{buzzi2025strongpositiverecurrenceexponential}, see also \cite[Proposition 11.5]{Sarig2013} and \cite[\S 6]{BenOvadia2018}}]\label{thm:BCS}
Let \(\Theta:E\to E\) be a \(\mathcal C^{1+\alpha}\) diffeomorphism on a compact smooth manifold \(E\). Let \(\mathcal X\) be either a Borel homoclinic class \(H\) or \(E\).
Then, for each \(\chi>0\), there exists a countable collection of sets $\mathscr R=\{R_i\}_{i\in I}$, which we call a $\chi$-Sarig partition, with the following properties:
\begin{itemize}
    \item  Each \(R_i\) is a  rectangle with local product structure  (see Definition \ref{def:rec});
    \item  whenever $x\in R_i$ and $\Theta(x)\in R_j$, then
$$\Theta\big(W^s(x,R_i)\big)\subset W^s(\Theta(x),R_j)
\ \text{and}\ 
\Theta\big(W^u(x,R_i)\big)\supset W^u(\Theta(x),R_j).$$
\end{itemize}

Let the topological Markov shift \(\Sigma=\Sigma(\mathscr G)\) be induced by the directed graph \(\mathscr G=(\mathscr V,\mathscr E)\) with \(\mathscr V:=\mathscr R\) and
\[
\mathscr E = \big\{\,R_i\to R_j;\ R_i\cap \Theta^{-1}(R_j)\neq\varnothing\,\big\}.
\]
Then:
\begin{enumerate}
\item[(1)] The map
\[
\pi:\ (R_i)_{i\in\mathbb Z}\in\Sigma\ \mapsto\ 
\bigcap_{n\in\mathbb N}\,\overline{\bigcap_{-n\le i\le n}\Theta^{-i}(R_i)}\in\mathcal X
\]
is well-defined and H\"older continuous. Moreover, \(\Theta\circ\pi=\pi\circ\sigma\) and \(\pi:\Sigma^\#\to\mathcal X\) is finite-to-one (see Definition~\ref{def:tms}).

\item[(2)] \(\Sigma\) is a locally compact countable-state Markov shift and \(\mu(\pi(\Sigma))=\mu(\pi(\Sigma^\#))=1\). If \(\mathcal X\) is a Borel homoclinic class, \(\Sigma\) can be chosen irreducible.

\item[(3)] Every ergodic \(\sigma\)-invariant probability measure \(\widehat\mu\) on \(\Sigma\) projects to an ergodic \(\Theta\)-invariant measure \(\mu:=\pi_*\widehat\mu\) with \(\mu(\mathcal X)=1\) and \(h_\mu(\Theta)=h_{\widehat\mu}(\sigma)\).

\item[(4)] For every \(\chi\)-hyperbolic \(\Theta\)-invariant probability measure \(\mu\) on \(E\) there exists a \(\sigma\)-invariant probability measure \(\widehat\mu\) on \(\Sigma\) such that \(\pi_*\widehat\mu=\mu\).
\end{enumerate}
\end{theorem}

We now state the main technical result on which our analysis is based. It refines Theorem~\ref{thm:BCS} in the setting of skew products.

\begin{theorem}\label{thm:1block}

  Let $\Theta:\Omega\times M \to \Omega\times M$ be a $\mathcal C^{1+\alpha}$ diffeomorphism such that
\begin{enumerate}
    \item[(a)]  $\Theta$ is a skew product i.e. $\Theta(\omega,x) = (\theta(\omega),T_\omega(x))$;
    \item[(b)] $\theta:\Omega\to\Omega$ is an Anosov map and $\mathbb P$ a $\theta$-invariant ergodic probability measure with full support.
\end{enumerate}
Let $\mathcal{X}\subset \Omega \times M$ be equal to a Borel homoclinic class $H$ or $\Omega\times M$. Then, for every $\chi>0$ there exists a locally compact topological Markov shift $\Sigma_s$, a subshift of finite type $\Sigma_b$ and Hölder continuous maps $\pi_s:\Sigma_s\to \mathcal{X}$ and $\pi_b:\Sigma_b\to \Omega$ such that
\begin{enumerate}
    \item $\pi_{s} \circ \sigma_s = \Theta \circ \pi_s$ and $\pi_b \circ \sigma_b = \theta\circ \pi_b$, where $\sigma_s$ is the left shift map on $\Sigma_s$ and $\sigma_b$ is the left shift map on $\Sigma_b$;
    \item  for every $\chi$-hyperbolic $\Theta$-invariant probability $\mu$ such that $\mu(\mathcal{X})=1$ with $(\proj_\Omega)_*\mu = \mathbb P$,
  one has ${\mu\big(\pi_s[\Sigma_s]\big)=\mu\big(\pi_s[\Sigma_s^{\#}]\big)=1}$;
    \item  the maps $\pi_s\big|_{\Sigma_s^{\#}}:\Sigma_s^{\#}\to\mathcal{X}$ and $\pi_b:\Sigma_b\to\Omega$ are finite-to-one;

    \item  there exists a $1$-block map  $\mathrm{proj}_{\Sigma_b}:\Sigma_s\to \Sigma_b$ such that $$\pi_b \circ \mathrm{proj}_{\Sigma_b} = \mathrm{proj}_\Omega\circ \pi_s.$$ Moreover, $\widehat{\mathbb P}\left[\mathrm{proj}_{\Sigma_b}(\Sigma_s)\right] = 1$ where $\widehat{\mathbb P}$ is the unique measure in $\mathcal M_e(\Sigma_b)$ such that $(\pi_b)_*\widehat{\mathbb P} = \mathbb P$.
    \item if \(\mathcal X\) is a Borel homoclinic class and there exists a $\chi$-hyperbolic \(\Theta\)-invariant ergodic probability measure \(\mu\) with \(\operatorname{supp}\mu\subset\mathcal X\) and \((\proj_\Omega)_*\mu=\mathbb P\), then the coding \((\Sigma_s,\sigma_s)\) can be chosen to be an irreducible topological Markov shift.

\end{enumerate}

Diagram \ref{diagram1} summarises all of the sets and projections used in the following proof.

\begin{diagram}[!ht]
\centering
\begin{tikzcd}
    \mathcal{X}\subset \Omega\times M \arrow[dddd, bend right=90, "{\rm proj}_\Omega"]\arrow[r, "\Theta"] & \mathcal{X}\subset \Omega\times M \arrow[dddd, bend left=90, "{\rm proj}_\Omega"]  \\
    \Sigma \arrow[r, "\sigma"] \arrow[u, "\pi"] & \Sigma \arrow[u, "\pi"]\\
    \Sigma_s \arrow[uu, bend left=70, "\pi_s"] \arrow[d, "{\rm proj}_{\Sigma_b}"] \arrow[r, "\sigma_s=\sigma_b\times\sigma"] \arrow[u, "{\rm proj}_\Sigma"]  &  \arrow[u, "{\rm proj}_\Sigma"]  \Sigma_s
    \arrow[uu, bend right=70, "\pi_s"] \arrow[d, "{\rm proj}_{\Sigma_b}"] \\
     \Sigma_b \arrow[r, "\sigma_b"] \arrow[d, "\pi_b"]&  \Sigma_b \arrow[d, "\pi_b"] \\
    \Omega \arrow[r, "\theta"] & \Omega
\end{tikzcd}
\caption{Commutative diagram of the maps $\Theta, \pi, \pi_s, \pi_b, \mathrm{proj}_\Sigma, \mathrm{proj}_{\Sigma_b}$, and $\mathrm{proj}_{\Omega}$. 
We note that $\Sigma_s\subsetneq \Sigma_b \times \Sigma,$ in fact,
$\Sigma_s \subset \{\un{(P,R)}; \un{P}\in \Sigma_b,\, \un{R}\in \Sigma\ \text{and }\pi_b(\un{P}) = \mathrm{proj}_\Omega \circ \pi(\un{R})\}$.}
\label{diagram1}
\end{diagram}

\end{theorem}

\begin{proof}[Proof of Theorem \ref{thm:1block}]

Let \(\chi>0\). Applying Theorem~\ref{thm:BCS} to the skew product $\Theta:\Omega\times M\to\Omega\times M$ and $\mathcal{X}$ be either equal to  Borel homoclinic class $H$ or $\Omega\times M$, we obtain a topological Markov shift $(\Sigma,\sigma)$ with the properties asserted therein, and let $\mathscr{R}$ be its $\chi$-Sarig partition. We assume $(\Sigma,\sigma)$ to be irreducible if $\mathcal{X}$ is a Borel homoclinic class. Since $\theta:\Omega\to\Omega$ is a transitive Anosov diffeomorphism and $\mathbb P$ is a measure with full support, there exists a subshift of finite type $(\Sigma_b,\sigma_b)$ associated with a Markov partition $\mathscr P$ as in Definition \ref{def:BowenPartition}. Since $\supp \mathbb P =\Omega$ we have that  $\mathbb P(\partial \mathscr P)=0$ (see Definition \ref{def:ba}).
Therefore 
\begin{align}
\mathbb P\left[\omega\in \Omega;\, \theta^i(\omega)\in \partial \mathscr P \ \text{for some }i\in\mathbb Z\right] &= \mathbb P\left[\bigcup_{i\in\mathbb Z} \theta^{-i}(\partial \mathscr P)\right]\nonumber\\
&\leq \sum_{i\in\mathbb Z} \mathbb P\left[\theta^{-i}(\partial \mathscr P)\right] = 0.\label{eq:meas0}
\end{align}
Let $\pi_b$ be as in \eqref{def:pib} and recall that
$$k_{\mathscr P} := \sup_{\omega\in \Omega} \# \pi_b^{-1}(\omega) <\infty. $$

Let $\mathscr G_s := (\mathscr V_s,\mathscr E_s)$ be the graph with vertices $$\mathscr V_s := \{(P,R);\, R\in\mathscr R, P\in\mathscr P\ \text{and }R\cap (\operatorname{int}P\times M) \neq \varnothing\},$$ and edges
$$\mathscr{E}_s = \left\{
  (P,R)\to (P',R') \;\middle|\;
  \begin{aligned}
  & (P,R),(P',R')\in \mathscr V_s,\ \text{and }\\
  &\Theta^{-1}[R'\cap (\operatorname{int}P'\times M)]\cap R\cap (\operatorname{int}P\times M)\neq \varnothing
  \end{aligned}
\right\}.$$
Define $(\Sigma_s,\sigma_s) := (\Sigma_s(\mathscr G_s),\sigma_s)$. Observe that for each $(P,R)\in \mathscr V_s$
$$\#\{(P',R') \in\mathscr V_s; (P,R)\to (P',R')\in \mathscr E_s \}\leq \#\{R' \in\mathscr V_s; R\to R'\in \mathscr E \} \cdot \# \mathscr P <\infty,$$
and
$$\#\{(P',R') \in\mathscr V_s; (P',R')\to (P,R)\in \mathscr E_s \}\leq \#\{R' \in\mathscr V_s; R'\to R\in \mathscr E \} \cdot \# \mathscr P <\infty.$$
Therefore, $\Sigma_s$ is a locally compact topological Markov shift.

Given $(P_i,R_i)_{i\in\mathbb Z}\in \Sigma_s$ we define
$$\pi_s ( (P_i,R_i)_{i\in\mathbb Z}) :=  \bigcap_{n\in \mathbb N} \overline{\bigcap_{-n\leq i\leq n}\Theta^{-i}[R_i \cap (\operatorname{int}P_i\times M)]}.$$

We divide the remainder of the proof into five steps.

\begin{step}[1]
We show that if given $m,n\in \mathbb Z$ with $n>m$ and a path 
$$(P_{m},R_{m})\to (P_{m+1},R_{m+1}) \to \ldots \to (P_n,R_n) $$
in $\mathscr G_s$ then
\begin{equation}
    \label{nonemptyexpr}
\bigcap_{i=m}^{n} \Theta^{-i}[R_{i} \cap (\operatorname{int}P_i\times M)] \neq \varnothing.
\end{equation}

In particular, we obtain that for each $(P_i,R_i)_{i\in\mathbb Z}\in \Sigma_s$
\begin{align*}
  \pi_s: \Sigma_s&\to \Omega\times M\\
  (P_{i},R_{i})_{i\in\mathbb Z}&\mapsto  \bigcap_{n\in \mathbb N} \overline{\bigcap_{-n\leq i\leq n}\Theta^{-i}\left[R_i \cap ( \operatorname{int}P_i\times M)\right] }
\end{align*}
is well defined and $\pi_s( (P_i,R_i)_{i\in\mathbb Z}) = \pi( (R_i)_{i\in\mathbb Z}).$
 \end{step}

We follow \cite[Lemma~12.1]{Sarig2013} and argue by induction on the length of the path.
Fix
$$
(P_m,R_m)\ \to\ (P_{m+1},R_{m+1})\ \to\ \cdots\ \to\ (P_n,R_n)\ \to\ (P_{n+1},R_{n+1}) .
$$
If \(n=m+1\) (i.e. a single edge from $m$ to \(m+1\)), the claim regarding the nonemptiness of the intersection in \eqref{nonemptyexpr} is immediate from the definition of the edge
\((P_m,R_m)\to(P_{m+1},R_{m+1})\).

Assume the statement holds up to \(n\), and prove it for \(n+1\).
By the induction hypothesis, there exists \((\omega_0,x_0)\) such that, writing
\((\omega_i,x_i):=\Theta^i(\omega_0,x_0)\),
$$
(\omega_i,x_i)\in R_i\cap\big(\operatorname{int}P_i\times M\big)\ \text{for all }i=m,\dots,n.
$$
Since \((P_n,R_n)\to(P_{n+1},R_{n+1})\), choose
$$
(\upsilon_n,y_n)\in\big(R_n\cap(\operatorname{int}P_n\times M)\big) \cap \Theta^{-1}\big(R_{n+1}\cap(\operatorname{int}P_{n+1}\times M)\big).
$$
Set
$
(\zeta_n,z_n):=[\,(\upsilon_n,y_n),(\omega_n,x_n)\,]_{\Theta}\in R_n,
$
which is well defined because \(R_n\) is a rectangle with local product structure. 
Since \(\Theta\) is a skew product and $(\upsilon_n,y_n),(\omega_n,x_n)\in R_n\cap(\operatorname{int} P_n \times M)$, the bracket decouples in base and fibre:
$$
[\,(\upsilon_n,y_n),(\omega_n,x_n)\,]_{\Theta}
=\big(\,[\upsilon_n,\omega_n]_{\theta},\, z_n\big),
$$
and in particular \([\zeta_n,\omega_n]_{\theta}\in \operatorname{int}P_n\).

Finally, define $(\zeta_0,z_0):=\Theta^{-n}(\zeta_n,z_n).$
 In the following, we show that
$$
(\zeta_0,z_0) \in \bigcap_{i=m}^{n+1}\Theta^{-i}\left[R_i\cap\big(\operatorname{int}P_i\times M\big)\right],
$$
which is sufficient to establish \eqref{nonemptyexpr}.

Define $(\zeta_i,z_i):= \Theta^i(\zeta_0,z_0)$ for each $i\in\{m,\ldots, n\}$. For clarity in what follows, we denote by $W^s_\Theta$ (resp.\ $W^u_\Theta$) the local stable (resp.\ unstable) manifold of $\Theta$, and by $W^s_\theta$ (resp.\ $W^u_\theta$) the local stable (resp.\ unstable) manifold of $\theta$.
  Using the Markov properties of the partitions $\mathscr R$ and $\mathscr P$, we obtain 
\begin{itemize}
    \item $\Theta^{n+1}(\zeta_0,z_0)=\Theta(\zeta_n,z_n)\in R_{n+1}\cap (\operatorname{int}P_{n+1}\times M)$, because
    $$
    \Theta[W^s_\Theta((\zeta_n,z_n), R_n)] \subset W^s_\Theta(\Theta(\zeta_n,z_n), R_{n+1}) \subset R_{n+1}
    $$
    and
    $$
    \theta[W^s_{\theta}(\zeta_n, \operatorname{int}P_{n})] \subset W^s_\theta(\theta(\zeta_n), \operatorname{int}P_{n+1}) \subset \operatorname{int}P_{n+1}.
    $$

    \item $\Theta^{n}(\zeta_0,z_0) = (\zeta_n,z_n)\in R_n\cap (\operatorname{int}P_{n}\times M)$ by construction.

    \item $\Theta^{n-1}(\zeta_0,z_0) = \Theta^{-1}(\zeta_{n},z_n)\in R_{n-1}\cap ( \operatorname{int}P_{n-1}\times M)$, because
    $$
    \Theta^{n}(\zeta_0,z_0)\in W^u_\Theta((\zeta_n,z_n), R_n) \subset R_n
    $$
    hence
    $$
    \Theta^{n-1}(\zeta_0,z_0) \in \Theta^{-1}[W^u_\Theta((\zeta_n,z_n), R_n)] \subset W^u_\Theta((\zeta_{n-1},z_{n-1}), R_{n-1}) \subset R_{n-1},
    $$
    and on the base,
    $$
    \theta^{n}(\zeta_0) \in W^u_\theta(\zeta_n, \operatorname{int}P_{n}) \subset \operatorname{int}P_{n}
    $$
    so
    $$
    \theta^{n-1}(\zeta_0) \in \theta^{-1}[W^u_\theta(\zeta_n, \operatorname{int}P_{n})] \subset W^u_\theta(\zeta_{n-1}, \operatorname{int}P_{n-1}) \subset \operatorname{int}P_{n-1}.
    $$

    \item $\Theta^{n-2}(\zeta_0,z_0)\in R_{n-2}\cap (\operatorname{int}P_{n-2}\times M)$, because from the previous item
    $$
    \Theta^{n-1}(\zeta_0,z_0)\in W^u((\zeta_{n-1},z_{n-1}), R_{n-1}) \subset R_{n-1}
    $$
    and therefore
    $$
    \Theta^{n-2}(\zeta_0,z_0) \in \Theta^{-1}[W^u_\Theta((\zeta_{n-1},z_{n-1}), R_{n-1})] \subset W^u_\Theta((\zeta_{n-2},z_{n-2}), R_{n-2}) \subset R_{n-2},
    $$
    with the analogous inclusion on the base:
    $$
    \theta^{n-2}(\zeta_0) \in \theta^{-1}[W^u_\theta(\zeta_{n-1}, \operatorname{int}P_{n-1})] \subset W^u_\theta(\zeta_{n-2}, \operatorname{int}P_{n-2}) \subset \operatorname{int}P_{n-2}.
    $$
\end{itemize}
Continuing in this way, we obtain
$$
\Theta^{n+1-k}(\zeta_0,z_0)\in R_{n+1-k}\cap (\operatorname{int}P_{n+1-k}\times M) \text{ for all } 0\le k\le n+1-m.
$$
which shows \eqref{nonemptyexpr}.

We now conclude Step~1. By \eqref{nonemptyexpr} we have
\begin{align*}
    \emptyset\neq \pi_s( \{(P_i,R_i)\}_{i\in\mathbb Z})&=  \bigcap_{n\in \mathbb N} \overline{\bigcap_{-n\leq i\leq n}\Theta^{-i}\left[R_i \cap ( \operatorname{int}P_i\times M)\right]}\\
    &\subset \bigcap_{n\in \mathbb N} \overline{\bigcap_{-n\leq i\leq n}\Theta^{-i}(R_i)} = \pi( (R_i)_{i\in \mathbb Z}).   
\end{align*}
Since  $\pi( (R_i)_{i\in \mathbb Z})$ is a singleton, it follows that $\pi_s( \{(P_i,R_i)\}_{i\in\mathbb Z}) =\pi( (R_i)_{i\in \mathbb Z}). $

\begin{step}[2] We show item $(1).$
    
\end{step}

Observe that the map $\pi_s$, defined above, is well defined and Hölder continuous. Indeed
\begin{align*}
    \mathrm{dist}_{\Omega\times M}( \pi_s\left( \{(P_i,R_i)\}_{i\in\mathbb Z}\right),  \pi_s\left (\{(P_i',R_i')\}_{i\in\mathbb Z} \right) &\leq \mathrm{dist}_{\Omega\times M}\left( \pi \left( (R_i)_{i\in\mathbb Z}\right),\pi \left((R_i')_{i\in\mathbb Z}\right) \right) \\
    &\leq C \mathrm{dist}_{\Sigma}((R_i)_{i\in\mathbb Z}, (R_i')_{i\in\mathbb Z})^{\eta}\\
    &\leq C\mathrm{dist}_{\Sigma_s}(\{(P_i,R_i)\}_{i\in\mathbb Z}, \{(P_i',R_i')\}_{i\in\mathbb Z})^{\eta},
\end{align*}
where $C$ and $\eta$ are the Hölder map constants of $\pi:\Sigma\to \Omega\times M$. Moreover, it is clear from the definitions of $\pi_s$ and $\pi_b$ that
\begin{align}
    \Theta \circ \pi_s = \pi_s \circ \sigma_s \ \text{and }\theta\circ \pi_b = \pi_b \circ \sigma_b. \label{eq:factor}
\end{align} 
Item $(1)$ now follows.

\begin{step}[3] We show that if $\mu$ is a $\chi$-hyperbolic, $\Theta$-invariant measure with $\mu(\mathcal X)=1$ and $(\mathrm{proj}_\Omega)_*\mu=\mathbb P$, then $\mu(\pi_s(\Sigma_s))=\mu(\pi_s(\Sigma_s^\#))=1$. Consequently, item~(2) follows.
\end{step}
Recall from Theorem \ref{thm:BCS} (2) that if $\mu(\mathcal{X})=1$, then  $\mu(\pi(\Sigma)) = \mu(\pi(\Sigma^\#)) =1$ for each $\chi$-hyperbolic $\Theta$-invariant measure. Since $\Sigma_s^\#\subset \Sigma_s$, and $(\proj_\Omega)_* \mu = \mathbb P$ we have from \eqref{eq:meas0} that
$$ \mu \left( (\omega,x) \in \Omega\times M;\, \pi(\un{R}) = (\omega,x)\ \text{for some }\un{R}\in \Sigma^\#\ \text{and }\omega\not\in \bigcup_{i\in\mathbb Z} \theta^{-i}(\partial \mathscr{P})  \right)=1.$$
Therefore, it is enough to show that for each 
$$\un{R}= (R_i)_{i\in\mathbb Z} \in \Sigma^\#\ \text{such that }\mathrm{proj}_\Omega \circ \pi(\un{R}) \not\in \bigcup_{i\in\mathbb Z} \theta^{-i}(\partial \mathscr{P}),$$ there exists $\un{P} = (P_i)_{i\in\mathbb Z}\in \Sigma_b$ such that $\un{(P,R)} = \{(P_i,R_i)\}_{i\in\mathbb Z} \in \Sigma_s^\#.$ Indeed, under this assumption, we obtain 
\begin{align*}
    \mu(\pi_s(\Sigma_s))&\geq \mu(\pi_s(\Sigma_s^\#))\\
    \geq&\mu \left( (\omega,x) \in \Omega\times M;\, \pi(\un{R}) = (\omega,x)\ \text{for some }\un{R}\in \Sigma^\#\ \text{and }\omega\not\in \bigcup_{i\in\mathbb Z} \theta^{-i}(\partial \mathscr{P})  \right)\\=&1.
\end{align*}
as required. 

Assume that $\un{R} \in \Sigma^\#$ and $\pi(\un{R}) = (\omega,x) \in \left(\Omega\setminus \bigcup_{i\in\mathbb Z} \theta^{-i}(\partial \mathscr{P})\right)\times M$. For each $i\in\mathbb Z$ let $P_i\in\mathscr P$ be such that $\theta^i(\omega) \in \operatorname{int}P_i.$
Since $\theta:\Omega\to \Omega$ is Anosov we have that
$$\bigcap_{-n \leq i\leq n}\theta^{-i}(\operatorname{int}P_i)\times M= \bigcap_{-n \leq i\leq n}\Theta^{-i}(\operatorname{int}P_i\times M)$$
is an open neighbourhood of $(\omega,x)$ in $\Omega\times M.$ Since $\pi(\un{R}) = (\omega,x)$ it follows that 
$$\bigcap_{-n\leq i\leq n} \Theta^{-i}(R_i) \cap V \neq \varnothing$$
for any open neighbourhood $V$ of $(\omega,x)$ in $\Omega\times M$. In particular,
$$ \varnothing\neq  \bigcap_{-n\leq i\leq n} \Theta^{-i}(R_i) \cap \bigcap_{-n\leq i\leq n}  \Theta^{-i} \left( \operatorname{int}  P_i \times M \right) =  \bigcap_{-n\leq i\leq n} \Theta^{-i}\left[R_i \cap(   \operatorname{int} P_i  \times M )\right].$$
Therefore, we obtain that $\{(P_i,R_i)\}_{i\in\mathbb Z} \in \Sigma_s.$ Since $\Sigma_b$ admits finitely many symbols and $\un{R}\in \Sigma^\#$, by the pigeonhole principle, it is clear that $ \{(P_i,R_i)\}_{i\in\mathbb Z}\in \Sigma_s^\#.$

\begin{step}[4] 
We show that the Lipschitz 1-block map
$$\mathrm{proj}_{\Sigma}:\Sigma_s\to\Sigma,\ 
\big((P_i,R_i)_{i\in\mathbb Z}\big)\mapsto (R_i)_{i\in\mathbb Z},$$
satisfies $\#\,\mathrm{proj}_{\Sigma}^{-1}(\underline{R})\le k_{\mathscr P}$ for every $\underline{R}\in\Sigma$, and item $(3)$. 
\end{step}

First of all, observe that given $\un{(P,R)}, \un{(P',R')}\in \Sigma_s$ we have that
\begin{align}
    \mathrm{dist}_{\Sigma} ( \mathrm{proj}_{\Sigma} (\un{(P,R}), \mathrm{proj}_{\Sigma} (\un{(P',R'}))) = \mathrm{dist}_{\Sigma} (\un{R}, \un{R'}) \leq   \mathrm{dist}_{\Sigma} (  \un{(P,R)},  \un{(P',R')})
\end{align}

Recall that $\pi_b:\Sigma_b\to \Omega$ is a finite-to-one surjective map over $\Omega$.  Given $\un{(P,R)}= (P_i,R_i)_{i\in\mathbb Z}\in \Sigma_s$, on the one hand we have that
$$\pi_s ( \un{(P,R)}) =  \bigcap_{n\in \mathbb N} \overline{\bigcap_{-n\leq i\leq n}\Theta^{-i}\left[R_i \cap (\operatorname{int}P_i\times M)\right]} \subset  \bigcap_{i\in  \mathbb Z} \theta^{-n}(P_n)\times M = \pi_b(\un{P})\times M.$$
On the other hand, we obtain that from  Step 1 we have that $\pi_s(\un{P,R}) = \pi(\un{R}).$ It follows that 
\begin{align*}
  \#\{\mathrm{proj}_{\Sigma}^{-1}(\un{R})\} &\leq \#\left\{\un{(P,R)} \in \Sigma_s; \proj_\Omega\circ \pi(\un{R}) = \pi_b(\un{P}) \right\}\\
   &= \sup_{\omega\in\Omega} \# \pi^{-1}(\omega)\leq \kappa_{\mathscr P}<\infty
\end{align*}

We claim that $\un{(P,R)}\in \Sigma_s^{\#}$ if, and only if, $\un{R} =\mathrm{proj}_{\Sigma}(\un{(P,R)}) \in \Sigma^\#.$ Indeed, if $\un{(P,R)}\in \Sigma_s^{\#}$, then by \eqref{Sigmahash} there exist $(P_1,R_1), (P_2,R_2)\in \mathscr V_s$, such that
$$\#\{i\in\mathbb N; (P,R)_i = (P_1,R_1) \} =  \#\{i\in\mathbb N; (P,R)_{-i} = (P_2,R_2) \}  =\infty.$$
Therefore $\un{R}\in\Sigma^\#.$ The proof of the reverse implication can be found in the last paragraph of Step 2. Hence,
$$\#\left\{\left(\pi_s|_{\Sigma_s^\#}\right)^{-1}(\omega,x)\right\} = \# \left\{ \mathrm{proj}_{\Sigma}^{-1}\left(\left(\pi|_{\Sigma^\#}\right)^{-1}(\omega,x)\right)\right\}.$$
Since the maps  $\mathrm{proj}_{\Sigma}$ and $\left.\pi\right|_{\Sigma^\#}$ are both finite-to-one, we obtain that  $$\#\left\{\left(\pi_s|_{\Sigma_s^\#}\right)^{-1} (\omega,x)\right\}<\infty.$$

Hence $\pi_s:\Sigma_s^{\#}\to \mathcal{X}$ is finite-to-one. Independently, since $\pi_b:\Sigma_b\to\Omega$ is obtained from a Markov partition for $\theta$, it is finite-to-one on $\Omega$. This proves item (3).

\begin{step}[5] We show that for any $\chi$-hyperbolic $\mu\in \mathcal M_e(\Theta)$ such that $(\mathrm{proj}_\Omega)_* \mu = \mathbb P$, there exists $\widehat{\mu}\in M_e(\sigma_s)$ such that $(\pi_s)_* \widehat{\mu} = \mu$.
\end{step}

This is done exactly as in \cite[Proof of Proposition 13.1]{Sarig2013}; we sketch the argument. By \cite[Step 1 of Proposition 13.1]{Sarig2013} the map
$$
F:(\omega,x)\in\pi_s(\Sigma_s^\#)\longmapsto \#\Big(\big.\pi_s^{-1}\big|_{\Sigma_s^\#}\{(\omega,x)\}\Big)\in\mathbb N
$$
is measurable. Moreover, there exists \(k=k_\mu\) such that \(F(\omega,x)=k\) for \(\mu\)-almost every \((\omega,x)\in\pi_s(\Sigma_s^\#)\) (\cite[Step 1 of Proposition 13.1]{Sarig2013}; see also \cite[Lemma 3.1]{Yoo2018}). Define the measure on \(\Sigma_s^\#\) by
$$
\widehat{\mu}(E):=\int_{\Omega\times M}\frac{1}{F(\omega,x)}
\left(\sum_{\pi_s\un{(P,R)}=(\omega,x)}\mathbbm{1}_E\big(\un{(P,R)}\big)\right)\, \mathrm d\mu(\omega,x),
$$
for any Borel set \(E\subset\Sigma_s^\#\). By \cite[Steps 3--4 of Proposition 13.1]{Sarig2013}, the measure \(\widehat{\mu}\) is
\begin{itemize}
    \item a well-defined \(\sigma_s\)-invariant probability measure; and
    \item such that \((\pi_s)_*\widehat{\mu}=\mu\) and \(h_{\widehat{\mu}}(\sigma_s)=h_\mu(\Theta)\).
\end{itemize}
Observe that, a priori, there is no need for \(\widehat{\mu}\) to be ergodic. The proof is finished by observing that each (typical) \(\sigma_s\)-invariant ergodic component \(\widehat{\mu}_i\) in the ergodic decomposition of \(\widehat{\mu}\) also satisfies \((\pi_s)_*\widehat{\mu}_i=\mu\) and \(h_{\widehat{\mu}_i}(\sigma_s)=h_\mu(\Theta)\) (see \cite[Step 5]{Sarig2013}).

\begin{step}[6] We show item (4).
\end{step}

Define
\[
\mathrm{proj}_{\Sigma_b}:\ \{(P_i,R_i)\}_{i\in\mathbb Z}\in\Sigma_s\ \mapsto\ \{P_i\}_{i\in\mathbb Z}\in\Sigma_b.
\]
Observe that \(\mathrm{proj}_{\Sigma_b}\) is a Lipschitz \(1\)-block map; in fact:
\begin{align}
    \mathrm{dist}_{\Sigma_b}\big(\mathrm{proj}_{\Sigma_b}(\un{(P,R)}),\mathrm{proj}_{\Sigma_b}(\un{(P',R')})\big)
    =\mathrm{dist}_{\Sigma_b}(\un{P},\un{P'})
    \leq \mathrm{dist}_{\Sigma_s}(\un{(P,R)},\un{(P',R')}).
\end{align}
Let \(\un{(P,R)}\in\Sigma_s\) and \(\pi_s(\un{(P,R)})=(\omega,x)\). Observe that, by construction, \(\pi_b(\un{P})=\omega\).
Hence
\[
\mathrm{proj}_{\Omega}\circ\pi_s(\un{(P,R)})=\mathrm{proj}_{\Omega}(\omega,x)=\omega=\pi_b(\un{P})
=\pi_b\circ\mathrm{proj}_{\Sigma_b}(\un{(P,R)}).
\]
Observe that, by construction, \(\proj_{\Sigma_b}\circ\sigma_s(\un{(P,R)})=\sigma_b(\un{P})=\theta\circ\proj_{\Sigma_b}(\un{(P,R)})\).

Let \(\mu\) be a \(\chi\)-hyperbolic measure on \(\Omega\times M\) such that \((\proj_\Omega)_*\mu=\mathbb P\). From Step 5, there exists \(\widehat{\mu}\in\mathcal M_e(\sigma_s)\) such that \((\pi_s)_*\widehat{\mu}=\mu\). We therefore obtain that
\begin{align*}
    \mathbb P&=(\proj_\Omega)_*\mu=(\proj_\Omega)_*(\pi_s)_*\widehat{\mu}=(\pi_b)_*\big[(\proj_{\Sigma_b})_*\widehat{\mu}\big].
\end{align*}
Defining the \(\sigma_b\)-invariant measure \(\widehat{\mathbb P}:=(\proj_{\Sigma_b})_*\widehat{\mu}\), we have \((\pi_b)_*\widehat{\mathbb P}=\mathbb P\). Since \(\mathbb P\) has full support, \eqref{eq:meas0} ensures that \(\mathbb P\big[\bigcup_{i\in\mathbb Z}\theta^{i}\partial\mathscr P\big]=0\), and therefore \(\widehat{\mathbb P}\) is the unique \(\sigma_b\)-invariant lift of \(\mathbb P\) with respect to the semi-conjugacy $\pi_b$ (see \cite[Theorem 3.18]{BowenBook}), which implies (4).

\begin{step}[7] We show item (5) and therefore conclude the proof of Theorem \ref{thm:1block}.
\end{step}

Let \(\mathcal X\) be a Borel homoclinic class. We restrict attention to the Borel homoclinic classes \(\mathcal X\) with the property that there exists a $\chi$-hyperbolic \(\Theta\)-invariant ergodic probability measure $\mu$ supported on \(\mathcal X\) such that $(\proj_\Omega)_* \mu = \mathbb P$. Recall that $\supp \mathbb P =\Omega$. To prove item \((5)\) it suffices to find a subshift \(\Sigma_s'\subset\Sigma_s\) such that
\begin{itemize}
    \item $\Sigma_s'$ is an irreducible topological Markov shift,
    \item $\mathrm{proj}_{\Sigma}(\Sigma_s') = \Sigma$, and
    \item $\widehat{\mathbb P}[\proj_{\Sigma_b}(\Sigma'_s)]=1$, where $\widehat{\mathbb P}$ is the unique measure in $\mathcal M_e (\Sigma_b)$ such that $(\pi_b)_*\widehat{\mathbb P} = \mathbb P$.
\end{itemize}

Since \(\mathcal{X}\) is a Borel homoclinic class, we may (and do) take \(\Sigma\) to be an irreducible two-sided topological Markov shift. Then \((\Sigma,\sigma)\) is topologically transitive and \(\sigma\) is a homeomorphism. Assume \(\Sigma\) is infinite (the periodic-orbit case is trivial).

Because \(\Sigma\) is a two-sided countable Markov shift with the standard metric (see Definition \ref{def:tms}), it is completely metrizable and has no isolated points. By transitivity, there exist $\un{R^{0}}$ with dense forward orbit and $\un{R^{1}}$ with dense backward orbit (see \cite[Proposition 1.1 and footnote 5]{SarigTF-TMS-2009}).

By \cite[Thm.~2.4]{Manoussos} applied to $\sigma:\Sigma\to\Sigma$ and to $\sigma^{-1}:\Sigma\to\Sigma$, we obtain that the sets
\[
G^{0}:=\left\{\un{Z}\in\Sigma:\; \un{R^{0}}\in \bigcap_{n\ge1}\overline{\bigcup_{k\ge n}\sigma^{k}(\un{Z})}\right\}\ \text{and }
G^{1}:=\left\{\un{Z}\in\Sigma:\; \un{R^{1}}\in \bigcap_{n\ge1}\overline{\bigcup_{k\ge n}\sigma^{-k}(\un{Z})}\right\}
\]
are dense \(G_{\delta}\) subsets of \(\Sigma\). By the Baire category theorem, \(G:=G^{0}\cap G^{1}\) is dense in $\Sigma$. Let $\un{R}$ be an element of $G$.

Observe that the set $\bigcap_{n\ge1}\overline{\bigcup_{k\ge n}\sigma^{k}(\un{R})}$ is closed, forward \(\sigma\)-invariant, and it contains \(\un{R^{0}}\). Hence $$\bigcap_{n\ge1}\overline{\bigcup_{k\ge n}\sigma^{k}(\un{R})} = \Sigma.$$
Applying the same argument to \(\sigma^{-1}\) we obtain that
\begin{align}
\bigcap_{n\ge1}\overline{\bigcup_{k\ge n}\sigma^{-k}(\un{R})}=\Sigma.\label{eq:sigmadense}
\end{align}
Therefore every point of $\Sigma$ is an accumulation point of both the forward and the backward $\sigma$-orbit of $\un{R}$ via $\sigma$.

Let \((\omega,x):=\pi(\un{R})\in\mathcal X\subset\Omega\times M\). We claim that $\omega\in \Omega\setminus\left(\bigcup_{i\in\mathbb Z}\theta^i(\partial\mathscr P)\right).$
Assume, for a contradiction, that the claim is false. Then, there exists \(i\in\mathbb Z\) with
\[
\theta^i(\omega)\in \partial\mathscr P=\partial^s\mathscr P\cup\partial^u\mathscr P.
\]

We treat the case \(\theta^i(\omega)\in\partial^s\mathscr P\); the unstable boundary case is analogous with backward time. Since the Borel homoclinic class \(\mathcal X\) admits a \(\Theta\)-invariant ergodic measure projecting to the full-supported ergodic measure \(\mathbb P\), by Theorem~\ref{thm:BCS}~(2) there exists \(\un R^*\in\Sigma\) with
\[
\pi(\un R^*)=(\omega^*,x^*)\ \text{and}\  \omega^*\in\Omega\setminus\partial^s\mathscr P.
\]
Because the forward orbit of \(\un R\) is dense, there is an increasing sequence \(k_j\to\infty\) such that $\sigma^{k_j}(\un R)\longrightarrow \un{R^*}$ as $j\to\infty$. By Theorem~\ref{thm:BCS}~(1), we have 
\begin{align}\mathrm{dist}_\Omega\big(\theta^{k_j}(\omega),\omega^*\big)
&\le \mathrm{dist}_{\Omega\times M}\big(\Theta^{k_j}(\omega,x),(\omega^*,x^*)\big)
= \mathrm{dist}_{\Omega\times M}\big(\Theta^{k_j}\pi(\un R),\pi(\un{R^*})\big) \nonumber\\
&\le C\,\mathrm{dist}_\Sigma\big(\sigma^{k_j}(\un R),\un {R^*}\big)^{\eta}\to 0\ \text{as }j\to\infty
\label{eq:Holder-goes-to-0}
\end{align}
for suitable constants \(C,\eta>0\).

On the other hand, by \cite[Prop.~3.5]{BowenBook},
\[
\theta(\partial^s\mathscr P)\subset \partial^s\mathscr P\ \text{and}\ 
\theta^{-1}(\partial^u\mathscr P)\subset \partial^u\mathscr P.
\]
Hence from \(\theta^i(\omega)\in\partial^s\mathscr P\) we get \(\theta^{i+n}(\omega)\in\partial^s\mathscr P\) for all \(n\ge0\). In particular, for all \(j\) large enough so that \(k_j\ge i\), $\theta^{k_j}(\omega)\in\partial^s\mathscr P.$ Since \(\partial^s\mathscr P\) is closed and \(\omega^*\notin\partial^s\mathscr P\), the distance $\delta:=\mathrm{dist}_\Omega\big(\omega^*,\partial^s\mathscr P\big)>0,$ and therefore $\mathrm{dist}_\Omega\big(\theta^{k_j}(\omega),\omega^*\big)\ge \delta$ for all such $j$ large enough.
This contradicts \eqref{eq:Holder-goes-to-0}. The case \(\theta^i(\omega)\in\partial^u\mathscr P\) is analogous. Hence
\(\omega\notin \bigcup_{i\in\mathbb Z}\theta^i(\partial\mathscr P)\), proving the claim.

Using that $\omega\notin \bigcup_{i\in\mathbb Z}\theta^i(\partial\mathscr P)$, from the proof of Step 3, there exists $\un{(P,R)}\in \Sigma_s$ such that $\pi_s(\un{(P,R)})=(\omega,x)$ and $\proj_{\Sigma}(\un{(P,R)}) = \un{R}$.  Let $\mathscr G'_s \subset \mathscr G_s$ be the maximal irreducible component of $\mathscr G_s$ containing $(P_0,R_0)$ (the $0$-th element of $\un{(P,R)}$) and define $\Sigma'_s := \Sigma(\mathscr G'_s) \subset \Sigma(\mathscr G_s) =  \Sigma_s.$ Observe that $\un{(P,R)}\in \Sigma_s'$ and therefore $\sigma^i_{s}(\un{(P,R)})\in \Sigma_s'$ for each $i\in \mathbb Z.$ 

We claim that $\operatorname{proj}_{\Sigma}(\Sigma'_s)=\Sigma$.
First, density: let $U\subset\Sigma$ be nonempty open. Since the orbit of $\un R$ is dense,
there exists $i\in\mathbb Z$ with $\sigma^i(\un R)\in U$. By construction, $\sigma^i(\un R)=\operatorname{proj}_{\Sigma}\left(\sigma_s^i(\un{(P,R)})\right)$ with $\sigma_s^i(\un{(P,R)})\in\Sigma'_s$, hence $U\cap \proj_{\Sigma}(\Sigma'_s)\neq\varnothing$. Thus
$\proj_{\Sigma}(\Sigma'_s)$ is dense.

To see that $\operatorname{proj}_{\Sigma}(\Sigma'_s)$ is closed, let $\{\un{R}^{(m)}\}_{m\ge1}\subset \proj_{\Sigma}(\Sigma_s')$ with $\un{R}^{(m)}\to \un{R}$ in $\Sigma$ as $m\to\infty$. For each $m$ choose a lift $\un{Y}^{(m)}=\{(P^{(m)}_j,R^{(m)}_j)\}_{j\in\mathbb Z}\in\Sigma_s'$ such that $\proj_{\Sigma}(\un{Y}^{(m)})=\un{R}^{(m)}$ (so $R^{(m)}_j=\un{R}^{(m)}_j$ for all $j$). Since $\Sigma_b$ is a subshift of finite type, it is compact. Therefore, passing through a subsequence if necessary, we can assume that $\un{P}^{(m)}:=\{P^{(m)}_j\}_{j\in\mathbb Z}\to \un{P}=\{P_j\}_{j\in\mathbb Z}$ in $\Sigma_b$ as $m\to\infty$. Because $\un{R}^{(m)}\to \un{R}$, for each fixed $j$ we have $R^{(m)}_j=R_j$ for all large $m$, and therefore,
$$
\un{Y}^{(m)}=\{(P^{(m)}_j,R^{(m)}_j)\}_{j\in\mathbb Z}\longrightarrow \un{Y}:=\{(P_j,R_j)\}_{j\in\mathbb Z}.
$$
The space $\Sigma_s'$ is a countable-state Markov shift (edge shift), hence closed; thus $\un{Y}\in\Sigma_s'$. Finally, $\proj_{\Sigma}(\un{Y})=\un{R}$, so $\un{R}\in \proj_{\Sigma}(\Sigma_s')$. This proves that $\proj_{\Sigma}(\Sigma_s')$ is closed.

In this way $\Sigma = \proj_{\Sigma} (\Sigma_s')$. As in Step 5, the same proof yields that $\widehat{\mathbb P}[\proj_{\Sigma_b}(\Sigma_s)]=1.$

\end{proof}

\section{Measures of maximal relative entropy on topological Markov shifts} 

This section is devoted to proving Theorem \ref{thm:dream} and Corollary \ref{cor:dream}. 
Since the results are formulated in an abstract shift setting, we adopt notation consistent with the literature on relative entropy for shifts. 

\label{sec:4}
\begin{definition}
    Given 
    \begin{enumerate}
        \item $(X,\sigma_X)$ and $(Y,\sigma_Y)$ shift spaces;
        \item $\pi: X\to Y$ a $1$-block map, $\nu(\pi(X)) =1$ and $\mathrm{supp}(\nu) = Y$;
        \item a $\sigma_Y$-invariant measure $\nu$ on $Y$;
        \item two $\sigma_X$-invariant measures $\mu_1$ and $\mu_2$ on $X$ with $\pi_*\mu_1 = \pi_*\mu_2 = \nu$;
    \end{enumerate}
we define the \emph{$\nu$-relatively independent joining of $\mu_1$ and $\mu_2$} as the unique measure on $X\times X$ satisfying\label{eq:rj}
\begin{align*}
    \mu_1\otimes_\nu \mu_2(A_1\times A_2) = \int_Y \mathbb E_{\mu_1}[\mathbbm 1_{A_1} \mid \pi^{-1}\mathcal B(Y) ]\circ \pi^{-1}(y) \cdot \mathbb E_{\mu_2}[\mathbbm 1_{A_2} \mid \pi^{-1} \mathcal B(Y) ]\circ \pi^{-1}(y)\, \nu(\d y).   
\end{align*}
for each $A_1,A_2\in\mathcal B(X)$.

In the case that $(X,\sigma_X)$ is a topological Markov shift and $\mu$ is a $\sigma_X$-invariant probability measure, we denote $h_{\mathrm{top}}(\sigma_X)$ as its  Gurevich topological entropy and $h_{\mu}(\sigma_X)$ its metric entropy, for details see \cite{SarigThermo} (see also, \cite{Gurevich1,Gurevich2}).  
\end{definition}

\begin{definition}[\cite{Quas}]\label{def:relorth}
We say that $\mu_1$ and $\mu_2$ are $\nu$-\emph{relatively orthogonal} if 
$$\mu_1\otimes_{\nu} \mu_2 \left(\{ (u,v)\in X\times X; u_0 = v_0\}\right) =0 $$
\end{definition}

\begin{notation}
    Given a topological Markov shift $(X,\sigma)$, we denote as $\mathcal A_X$ its alphabet. Also, given $x\in X$ we denote as $x_i$ its $i$-th coordinate.  
  Let $a\in A_X$, we define the $1$-cylinder at coordinate $k$ as $[a]_{k} := \{x \in X; x_k = a\}.$
\end{notation}

The purpose of this section is to show the following theorem.
\begin{theorem}\label{thm:dream}
Let $(X,\sigma_X)$ be an irreducible countable topological Markov shift with finite Gurevich topological entropy,  $(Y,\sigma_Y)$ an irreducible subshift of finite type and $\pi:X\to Y$ a $1$-block map such that $\nu(\pi(X)) =1$ and $\mathrm{supp}(\nu) = Y$. Then, if $\mu_1$ and $\mu_2$ are ergodic measures of maximal $\nu$-relative entropy then  $\mu_1$ and $\mu_2$ are $\nu$-relatively orthogonal.
\end{theorem}

Before proving Theorem \ref{thm:dream} we highlight the following useful corollary:
\begin{corollary}\label{cor:dream}
  Let $(X,\sigma_X)$ be an irreducible topological Markov shift with finite Gurevich topological entropy, $(Y,\sigma_Y)$ an irreducible subshift of finite type and $\pi:X\to Y$ a $1$-block map such that $\nu(\pi(X)) =1$ and $\mathrm{supp}(\nu) = Y$. Then, $(X,\sigma_X)$ can only admit countably many ergodic measures of maximal $\nu$-relative entropy.
\end{corollary}
\begin{proof}
Assume that $(X,\sigma_X)$ admits uncountably many $\nu$-relative measures of maximal entropy $\{\mu_i\}_{i\in\mathcal I}$.
Since $\mathcal I$ is uncountable while the alphabet $\mathcal A_X$ is countable, there exist $a\in\mathcal A_X$ and an uncountable subset $\mathcal I_0\subset\mathcal I$ such that
$$
\mu_i([a]_0)>0 \text{ for every } i\in\mathcal I_0 .
$$

For each $i\in\mathcal I_0$ we have that
$$\mu_i([a]_0) = \int_Y \mathbb E_{\mu_i}\left[\mathbbm{1}_{[a]_0}\mid\pi^{-1}\mathcal B(Y)\right]\circ \pi^{-1}(y)\,\nu(\mathrm dy) > 0.$$
Let $g_i(y):=\mathbb E_{\mu_i}\left[\mathbbm{1}_{[a]_0}\mid\pi^{-1}\mathcal B(Y)\right] (\pi^{-1}(y))$;  then with $A_i:=\{y\in Y;\, g_i(y)>0\}$ we have $\nu(A_i)>0$.

By Theorem~\ref{thm:dream}, for any distinct $i_1,i_2\in\mathcal I_0$ the $\nu$-relatively independent joining $\mu_{i_1}\otimes_\nu \mu_{i_2}$ gives zero mass to the coincidence event $\{(u,v)\in X\times X;\, u_0=v_0\}$; in particular,
$$
\mu_{i_1}\otimes_\nu \mu_{i_2}\bigl(\{(u,v);\, u_0=v_0=a\}\bigr)=0 .
$$
From the definition of $\mu_{i_1}\otimes_\nu \mu_{i_2}$ given in Definition \ref{eq:rj} we obtain that $
\nu(A_{i_1}\cap A_{i_2})=0.$ Hence the family $\{A_i\}_{i\in\mathcal I_0}$ is pairwise disjoint modulo $\nu$-null sets.

Moreover,
$$
\mathcal I_0 = \bigcup_{n\in\mathbb N}\Bigl\{ i\in\mathcal I_0;\, \nu(A_i)>\tfrac1n \Bigr\}.
$$
For some $n$ the set $\mathcal I_1:=\{ i\in\mathcal I_0;\, \nu(A_i)>\tfrac1n \}$ is uncountable. But in a probability space there can be at most $n$ pairwise disjoint measurable sets of measure $>1/n$, a contradiction. Therefore $(X,\sigma_X)$ cannot admit uncountably many $\nu$-relative measures of maximal entropy.

\end{proof}

From now on, we fix $(X,\sigma_X)$ as an irreducible topological Markov shift with countable alphabet, $(Y,\sigma_Y)$ as an irreducible subshift of finite type, and $\pi: X\to Y$ a continuous $1$-block map such that $\nu(\pi(X))=1$.

In the following, we will construct a subshift of finite type $(X_n,\sigma_{X_n})$ which, in some sense, approximates $(X,\sigma_X)$. Since $\pi$ is a $1$-block map we have that $\gamma := \{\pi^{-1}([i]_0)\in\mathcal B(X); i\in\mathcal A_Y\}$ is a partition of $X$ such that each $\gamma$ is a union of $1$-cylinders at coordinate $0$. For each $n\in\mathbb N$ we define the partition of $X$ as
$$\alpha_n := \left\{[a_1]_0,[a_2]_0,\ldots, [a_n]_0, X \setminus \left(\bigcup_{i=1}^{n}[a_i]\right)\right\} \bigvee \gamma, $$
where $\{a_i; i\in\mathbb N\}$ is a fixed enumeration of $\mathcal A_X$. Observe that each element of $\alpha_n$ is the union of $1$-cylinders at coordinate $0$. 

Using the partition $\alpha_n$, we define a subshift of finite type $(X_n,\sigma_{X_n})$ as follows.
Let \(\mathscr{G}_n = (\mathscr V_n,\mathscr E_n)\) be the directed graph with vertex set $\mathscr V_n := \alpha_n$ and edge set

$$\mathscr E_n := \bigl\{ A\to A';\, A,A'\in\alpha_n\ \text{and } A \cap \sigma_X^{-1}(A') \neq \varnothing \bigr\}.$$
We then set $X_n := \Sigma(\mathscr{G}_n)$.

Define \(\mathrm{proj}_{X_n}:X\to X_n\) as the canonical $1$-block code associated with $\alpha_n$:
$$(\mathrm{proj}_{X_n}(x))_k = A \ \text{whenever}\ \sigma_X^k(x)\in A,\ A\in\alpha_n,\ k\in\mathbb Z.$$

Also, since the partition $\alpha_n$ refines the partition $\gamma$, for each $A\in\alpha_n$ there exists a unique $i(A)\in \{[i_0];i_0\in\mathcal A_Y\}$ such that $\pi(A)\subset [i(A)]_0 \subset Y.$ This labels each atom $A\in\alpha_n$ by $i(A)$ and induces a $1$-block map $\pi_n:X_n\to Y$ (induced by $\pi$ defined by
$$\pi_n:(A_k)_{k\in \mathbb N}\in X_n \mapsto (i(A_k))_{k\in\mathbb N} \in Y. $$
Observe that for each $n\in\mathbb N$ we have that $\pi = \pi_n \circ \mathrm{proj}_{X_n}.$  

Let \(\mu\) be a \(\sigma\)-invariant measure on \(X\) with \(\pi_*\mu=\nu\), and for each \(n\in\mathbb N\) set \(\mu_n=(\mathrm{proj}_{X_n})_*\mu\). Since \(\pi=\pi_n\circ\mathrm{proj}_{X_n}\), we have
\[
(\pi_n)_*\mu_n=(\pi_n\circ\mathrm{proj}_{X_n})_*\mu=\pi_*\mu=\nu.
\]
Hence \(\nu\) is supported on \(\pi_n(X_n)\), i.e. \(\nu\big(Y\setminus\pi_n(X_n)\big)=0\). Because \(\pi_n(X_n)\) is closed in \(Y\) and \(\supp(\nu)=Y\), it follows that \(\pi_n(X_n)=Y\). Therefore \(\pi_n:X_n\to Y\) is surjective for every \(n\in\mathbb N\).

From the Kolmogorov-Sinai theorem (see \cite[Theorem 4.9]{WaltersBook}) we have that
$$h_{\mu_n}(X_n) = H_{\mu_n}\left(\beta_n \,\left|\, \bigvee_{i=1 }^{\infty} \sigma^{-i}_{X_n} (\beta_n)  \right. \right) = H_{\mu}\left(\alpha_n\,\left|\, \bigvee_{i=1 }^{\infty} \sigma^{-i}_X (\alpha_n)\right. \right)$$
where $\beta_n$ is the partition of $X_n$ $1$-cylinders at coordinate $0$. Since, $\alpha_1\preccurlyeq\alpha_2\preccurlyeq \ldots$ and $\mathcal B(X)= \mathcal B\left( \bigvee_{n\in\mathbb N}\alpha_n\right)$ from  \cite[Theorem 4.14]{WaltersBook}
$$h_\mu (\sigma_X) = \lim_{n\to\infty}  h_{\mu_n}(\sigma_{X_n}).$$
With a slight abuse of notation, we write $a_1,\ldots,a_n\in\mathcal A_{X_n}$ whenever the context is clear.

\subsection{Petersen-Quas-Shin construction}
\label{sec:construction}

Select and fix $a\in A_X$. Choose $n$ large enough that  $[a]_0\in\alpha_n $ (so $a\in A_{X_n}$).  Let $\mu_n^1,\mu_n^2$ be two distinct ergodic $\sigma_{X_n}$-invariant measures on $X_n$ with $(\pi_n)_*\mu_n^1=(\pi_n)_*\mu_n^2=\nu,$ and assume \(\mu_n^1\) and \(\mu_n^2\) are not \(\nu\)-relatively orthogonal (Definition \ref{def:relorth}).
Following \cite{Quas}, there exists an invariant measure \(\mu_n^3\) of $(X_n,\sigma_{X_n})$ with \((\pi_n)_*\mu_n^3=\nu\) and
\[
h_{\mu_n^3}(\sigma_{X_n}) > \tfrac12\left(h_{\mu_n^1}(\sigma_{X_n}\right)+h_{\mu_n^2}\left(\sigma_{X_n})\right).
\]

We sketch the construction of $\mu_n^3$. Diagram \ref{diagram2} summarises all semi-conjugacies used in this construction in this section. Recall, from the previous section, that $\pi_n:X\to Y$ is surjective for every $n\in\mathbb N$.  Let
$
Z_n := X_n \times X_n \times \{0,1\}^{\mathbb Z}
$
and define the map $T := \sigma_{X_n}\times\sigma_{X_n}\times\sigma_{\mathrm{Ber}}$, where $\sigma_{\mathrm{Ber}}$ is the shift on $\{0,1\}^{\mathbb Z}$.
Let $\zeta$ be the $(1/2,1/2)$–Bernoulli measure on $\{0,1\}^{\mathbb Z}$. Also, set $\widehat{\mu}_n := \mu_n^1 \otimes_\nu \mu_n^2$ and  $\eta_n = \widehat{\mu}_n  \otimes \zeta$. Observe that $\eta_n$ is $T$–invariant.

Define three maps $\proj_{X_n}^1,\proj_{X_n}^2,\proj_{X_n}^3: Z_n \to X_n$ by
\begin{enumerate}
  \item[(i)] $\proj_{X_n}^1(u,v,r) := u$,
  \item[(ii)] $\proj_{X_n}^2(u,v,r) := v$,
  \item[(iii)] for $\proj_{X_n}^3$, first set, for each $k\in\mathbb Z$,
   \begin{align}
  N_k^n(u,v) := \sup\{m<k ; u_m = v_m=a\},\, \text{with the convention } \sup\varnothing = -\infty.\label{eq:Nn}
  \end{align}
   \( r\in\{0,1\}^{\mathbb Z} \), \( r_n \) denotes its \( n \)-th coordinate, and let
  \begin{align}
      r\in\{0,1\}^\mathbb Z\mapsto r_{-\infty}\in\{0,1\}\label{eq:rinfty}
  \end{align} be an independent Bernoulli\( (\tfrac12,\tfrac12) \) random variable. For \( k\in\mathbb Z \) define
$$
\bigl(\mathrm{proj}^3_{X_n}(u,v,r)\bigr)_k:=\begin{cases}
u_k,&\text{if }r_{N^n_k(u,v)}=0,\\
v_k,&\text{if }r_{N^n_k(u,v)}=1.
\end{cases}
$$

\end{enumerate}
By construction, $\proj_{X_n}^3(u,v,r)\in X_n$: between successive coincidence times of $u$ and $v$ we copy a block from either $u$ or $v$, and since $X_n$ is a topological Markov shift, switching only at coincidence times preserves admissibility. From the definition of $\widehat{\mu}_n$, we have that $$\pi_n(u) = \pi_n(v)\mbox{ for } \widehat{\mu}_n\text{-almost every $(u,v)\in X_n\times X_n$}.$$ Since $\pi_n$ is 1-block map we obtain that
\begin{align}
  \pi_n \circ \proj_{X_n}^3(z) = \pi_n \circ \proj_{X_n}^2(z) =  \pi_n \circ \proj_{X_n}^1(z) \ \text{for }\eta_n\text{-almost every}\ z\in Z_n.\label{eq:pin}  
\end{align}
Observe that by construction $\mu_n^1 = (\proj_{X_n}^1)_*\eta_n$ and $\mu_n^2 = (\proj_{X_n}^2)_*\eta_n$. We therefore define  
$$\mu_n^3 = (\proj_{X_n}^3)_*\eta_n,$$
Since $\proj_{X_n}^3:Z_n\to X_n$ is a $1$-block factor map, $\mu_n^3$ is $\sigma_{X_n}$-invariant.

Now, we quote a result from \cite{Quas} which is contained in the proof of Theorem 1 (see \cite[Equation (45)]{Quas}).

\begin{theorem}[\cite{Quas}]\label{thm:quasin}
Under the above assumptions and notations, for any $j_*\in\mathcal A_{X_n}$ for each $n\in\mathbb N$
$$
h_{\mu_n^3}(\sigma_{X_n}\mid \nu)
\ge
\frac{h_{\mu_n^1}(\sigma_{X_n}\mid \nu) + h_{\mu_n^2}(\sigma_{X_n}\mid \nu)}{2} + \int_{S_n^{(a)}} \Xi_n^{j_*}(z)\, \eta_n(\mathrm{d}z),
$$
where
$$
S_n^{(a)} := \left\{ z\in Z_n;\, (\proj_{X_n}^1 (z))_{-1} = (\proj_{X_n}^2 (z))_{-1}=a \,\right\}.
$$
The integrand $\Xi_n^{j_*}(z)$ is given by
$$
\Xi_n^{j_*}(z)
= \psi\left(\frac{A_{n}^{j_*}(z)+B_{n}^{j_*}(z)}{2}\right)
 - \frac{1}{2}\psi\big( A_{n}^{j_*}(z)\big)
 - \frac12  \psi\big(B_{n}^{j_*}(z)\big),
$$
with $\psi(t):= -t \log t$, and
$$
A_{n}^{j_*}(z):=\mathbb{E}_{\eta_n}\big(\mathbbm{1}_{[j_*]_0}\circ \proj_{X_n}^1 \,\big|\, \mathcal{B}_{1,n}^- \vee \mathcal{B}_{0,n}\big)(z),\ 
B_{n}^{j_*} (z):=\mathbb{E}_{\eta_n}\big(\mathbbm{1}_{[j_*]_0}\circ \proj_{X_n}^2 \,\big|\, \mathcal{B}_{2,n}^- \vee \mathcal{B}_{0,n}\big)(z),
$$
where 
$$
\mathcal{B}_{1,n}^- := \sigma\left(\,\{z\in Z_n ;\  (\proj_{X_n}^1 (z))_{-m}=j \,\} ;\, j\in\mathcal A_{X_n},\ m\in \mathbb N\right),
$$
$$
\mathcal{B}_{2,n}^- := \sigma\left(\,\{z\in Z_n ;\  (\proj_{X_n}^2 (z))_{-m}=j \,\};\, j\in\mathcal A_{X_n},\ m\in \mathbb N\right),
$$
\begin{align}
    \mathcal B_{0,n} := (\pi_n\circ\proj_{X_n}^1)^{-1}(\mathcal B(Y)).\label{eq:B0}
\end{align}
\end{theorem} 

\begin{remark}
    Observe that by \eqref{eq:pin} one has 
$$(\pi_n\circ\proj_{X_n}^1)^{-1}(\mathcal B(Y)) = (\pi_n\circ\proj_{X_n}^2)^{-1}(\mathcal B(Y)) = (\pi_n\circ\proj_{X_n}^3)^{-1}(\mathcal B(Y))\ (\mathrm{mod }\ \eta_n).$$
Thus, we may alternatively define $\mathcal B_{0,n}:=(\pi_n\circ\proj_{X_n}^i)^{-1}(\mathcal B(Y))$ in \eqref{eq:B0} using any $i\in\{1,2,3\}$, and the conclusions of Theorem \ref{thm:quasin} hold.
\end{remark}

\begin{diagram}[!ht]
\[\begin{tikzcd}
	{X\times X} & {Z=X\times  X\times\{0,1\}^\mathbb Z} && X \\
	&&&& Y \\
	X & {Z_n=X_n\times X_n\times\{0,1\}^\mathbb Z} && {X_n}
	\arrow["{p_X^i}"', from=1-1, to=3-1]
	\arrow["{\proj_X^i}", from=1-2, to=1-4]
	\arrow["{\proj_{Z_n}}"', from=1-2, to=3-2]
	\arrow["\pi", from=1-4, to=2-5]
	\arrow["{\proj_{X_n}}"', from=1-4, to=3-4]
	\arrow["{\proj_{X_n}^i}", from=3-2, to=3-4]
	\arrow["{\pi_n}"', from=3-4, to=2-5]
\end{tikzcd}\]
\caption{Diagram of the maps $p_X^i,\proj_{Z_n}$, $\proj_{X_n}$, $\pi$, $\pi_n$, $\proj_{X}^i$ and $\proj^i_{X_n}$ with $i\in\{1,2,3\}$. The right diagram commutes, and all maps are semi-conjugacies between the shift maps on the respective spaces.}\label{diagram2}
\end{diagram}

\begin{proof}[Sketch of the proof of Theorem \ref{thm:quasin}]

We follow closely \cite[Proof of Theorem 1]{Quas}. Let $\mathcal P$ be the partition into $0$ cylinders of $X_n$ and $\mathcal P_i = (\proj_{X_n}^i)^{-1}\mathcal P$ for $i\in\{1,2,3\}$. For simplicity of the notation, since $n$ is fixed we write $\mathcal B_1^- := \mathcal B_{1,n}^-, \mathcal B_2^- := \mathcal B_{2,n}^-,$
$$\mathcal B_3^- = \mathcal B_{3,n}^- = \sigma\left(\left\{z\in Z_n; (\proj_{X_n}^3(z))_{-m}=j\right\}; j\in \mathcal A_{X_n}\ \text{and }m\in\mathbb N\right)  $$
and $\mathcal B_{0,n} := \mathcal B_0.$ 

First of all, for each $i\in\{1,2,3\}$ we have that (see \cite[Equation (33)]{Quas})
$$h_{\mu_n^i}(\sigma_{X_n}) = H_{\eta_n}\left(\mathcal P_i\mid \mathcal B_i^-\vee \mathcal B_0\right) + h_\nu(Y).  $$
By definition,
\begin{align}
 H_{\eta_n}(\mathcal P_i \mid \mathcal B_i^- \vee \mathcal B_0) &= \int_{Z_n} -\sum_{j\in\mathcal A_{X_n}} \mathbbm 1_{[j]_0}\circ \proj_{X_n}^i  \log \left(\mathbb E_{\eta_n}\left[\mathbbm 1_{[j]_0}\circ \proj_{X_n}^i  \mid \mathcal B_i^-\vee \mathcal B_0\right]  \right) \d \eta_n\nonumber\\
&\hspace{-2cm} =  \int_{Z_n}-\sum_{j\in\mathcal A_{X_n}} \mathbb E_{\eta_n}\left[\mathbbm 1_{[j]_0}\circ \proj_{X_n}^i\mid \mathcal B_i^-\vee \mathcal B_0\right]  \log\left( \mathbb E_{\eta_n}\left[\mathbbm 1_{[j]_0}\circ \proj_{X_n}^i \mid \mathcal B_i^-\vee \mathcal B_0\right]\right)   \d \eta_n\nonumber\\
  &= \int_{Z_n}\sum_{j\in\mathcal A_{X_n}}\psi\left(\mathbb E_{\eta_n}\left[\mathbbm 1_{[j]_0}\circ \proj_{X_n}^i \mid \mathcal B_i^-\vee \mathcal B_0\right]\right)  \d \eta_n,\label{eq:eqentr}
\end{align}
where $\psi(t) = -t \log (t).$

Repeating verbatim the computation in \cite[Proof of identity (35)]{Quas}, we obtain that for any $j\in\mathcal A_{X_n}$ the function $$g_n^j({z}):= \mathbb{E}_{\eta_n}\big(\mathbf{1}_{[j]_0}\circ \proj_{X_n}^3\mid \mathcal{B}_1^-\vee \mathcal{B}_2^-\vee \mathcal{B}_3^-\vee \mathcal{B}_0\big)({z}),$$ 
where ${z}= ({u},{v},{r}),$ satisfies

\begin{itemize}
\item If $(\proj_{X_n}^3(z))_{-1}=u_{-1}\neq v_{-1}$, then
\[
g_n^j(z)
=
\mathbb{E}_{\eta_n}\big(\mathbf{1}_{[j]_0}\circ \proj_{X_n}^1\mid \mathcal{B}_1^-\vee \mathcal{B}_0\big)(z)= A^j_n(z).
\]

\item If $(\proj_{X_n}^3(z))_{-1}=v_{-1}\neq u_{-1}$, then
\[
g_n^j(z)
=
\mathbb{E}_{\eta_n}\big(\mathbf{1}_{[j]_0}\circ \proj_{X_n}^2\mid \mathcal{B}_2^-\vee \mathcal{B}_0\big)(z)=B^j_n(z).
\]

\item If $u_{-1}=v_{-1}=a$, then
\begin{align*}
g_n^j(z)
&=
\frac12\,\mathbb{E}_{\eta_n}\big(\mathbf{1}_{[j]_0}\circ \proj_{X_n}^1\mid \mathcal{B}_1^-\vee \mathcal{B}_0\big)(z)
+\frac12\,\mathbb{E}_{\eta_n}\big(\mathbf{1}_{[j]_0}\circ \proj_{X_n}^2\mid \mathcal{B}_2^-\vee \mathcal{B}_0\big)(z)\\
&=\frac{1}{2}\left(A_{n}^{j}(z)+B_{n}^{j}(z)\right)
\end{align*}

\item If $u_{-1}=v_{-1}\neq a$, then
$$g_n^j(z)
=
\mathbb{E}_{\eta_n}\big(\mathbf{1}_{[j]_0}\circ \proj_{X_n}^1\mid \mathcal{B}_1^-\vee \mathcal{B}_0\big)(z)\ \text{when }r_{N_{-1}^n(u,v) }=0$$
and
$$g_n^j(z) = \mathbb{E}_{\eta_n}\big(\mathbf{1}_{[j]_0}\circ \proj_{X_n}^2\mid \mathcal{B}_2^-\vee \mathcal{B}_0\big)(z)\ \text{when } r_{N_{-1}^n(u,v)} = 1.$$
\end{itemize}
From \eqref{eq:eqentr} we obtain that
\begin{align}
    H_{\eta_n} (\mathcal P_3 \mid \mathcal B_3^-\vee \mathcal B_0)&\geq  H_{\eta_n} (\mathcal P_3 \mid \mathcal B_1^-\vee B_2^-\vee B_3^- \vee\mathcal B_0)\\
    &= \int_{Z_n} \sum_{j\in \mathcal A_{X_n}} \psi \circ g_n^j(z)\, \d \eta_n(z) \label{eq:des0}
\end{align}
We define the sets
\begin{itemize}
    \item $S_1 = \{z=(u,v,r)\in Z_n;\, (\proj_{X_n}^3(z))_{-1} = u_{-1} \neq v_{-1}  \}$;
    \item $S_2 = \{z=(u,v,r)\in Z_n;\, (\proj_{X_n}^3(z))_{-1} = v_{-1} \neq u_{-1}  \}$;
    \item $S_3^{(a)} = \{z=(u,v,r)\in Z_n; u_{-1} = v_{-1} = a  \}$; and
    \item $S_3^{(\neq a)} = \{z=(u,v,r)\in Z_n; u_{-1} = v_{-1} \neq a \}$.
\end{itemize}

As in \cite[Equations (41), (42) and (43)]{Quas} we have that
\begin{align}
\int_{S_1\cup S_2} \psi \circ g_n^j(z) \d \eta_n(z) =& \frac{1}{2} \int_{S_1\cup S_2}   \psi\big(\mathbb{E}_{\eta_n}\big(\mathbf{1}_{[j]_0}\circ \proj_{X_n}^1\mid \mathcal{B}_1^-\vee \mathcal{B}_0)\big)(z)\ \d \eta_n(z)\nonumber\\
&+\frac{1}{2} \int_{S_1\cup S_2}    \psi\big( \mathbb{E}_{\eta_n}\big(\mathbf{1}_{[j]_0}\circ \proj_{X_n}^2\mid \mathcal{B}_2^-\vee \mathcal{B}_0)\big)(z) \  \d \eta_n(z) \label{eq:des1}
\end{align}
Since $\eta_n =\mu_n^1\otimes_\nu \mu_n^2 \otimes \zeta$ and $\zeta$ is $\left(\frac{1}{2},\frac{1}{2}\right)$-Bernoulli we have that 
$$\zeta( r_{N_{-1}^n(u,v)} =0) = \zeta(r_{N_{-1}^n(u,v)} =1)=\frac{1}{2}$$
and we obtain that
\begin{align}
\int_{
S_3^{(\neq a)}} \psi \circ g_n^j\, \d \eta_n =& \frac{1}{2} \int_{
S_3^{(\neq a)}}   \psi\circ \mathbb{E}_{\eta_n}\big(\mathbf{1}_{[j]_0}\circ \proj_{X_n}^1\mid \mathcal{B}_1^-\vee \mathcal{B}_0\big)\, \d\eta_n\nonumber\\
&+\frac{1}{2} \int_{
S_3^{(\neq a)}}    \psi\circ \mathbb{E}_{\eta_n}\big(\mathbf{1}_{[j]_0}\circ \proj_{X_n}^2\mid \mathcal{B}_2^-\vee \mathcal{B}_0\big)\, \d\eta_n.\label{eq:des2}
\end{align}
Furthermore,
\begin{align}
\int_{S_3^{(a)}} \psi \circ g_n^j(z)\, \d \eta_n&=  \int_{
S_3^{(a)}}   \psi\left(\frac{1}{2}\left(A_{n}^{j} +B_{n}^{j}\right) \right)\,\d \eta_n\\
&=  \int_{
S_3^{(a)}} \frac12 \psi\circ  A_{n}^{j} \, \d \eta_n + \int_{
S_3^{(a)}} \frac12  \psi\circ B_{n}^{j} \d \eta_n  +\int_{
S_3^{(a)}}  \Xi_n^j\,\d\eta_n .\label{eq:des3}
\end{align}

Combining \eqref{eq:des0}, \eqref{eq:des1}, \eqref{eq:des2} and \eqref{eq:des3} we obtain that
\begin{align*}
    h_{\mu_3^n}(\sigma_{X_n}\mid\nu) =& H_{\eta_n}(\mathcal P_3 \mid \mathcal B_3^- \vee \mathcal B_0) \geq \int_{Z_n} \sum_{j\in \mathcal A_{X_n}} \psi \circ g_n^j\, \d \eta_n \\
    =& \int_{S_1\cup S_3 \cup S_3^{(\neq a)}} \sum_{j\in \mathcal A_{X_n}} \psi \circ g_n^j \d \eta_n + \int_{S_3^{(a)}} \sum_{j\in \mathcal A_{X_n}} \psi \circ g_n^j\, \d \eta_n\\
    \geq &\frac{1}{2} \int_{
X_n}   \psi\circ \mathbb{E}_{\eta_n}\big(\mathbf{1}_{[j]_0}\circ \proj_{X_n}^1\mid \mathcal{B}_1^-\vee \mathcal{B}_0\big) \,\d \eta_n\nonumber\\
 &+\frac{1}{2} \int_{X_n}    \psi\circ \mathbb{E}_{\eta_n}\big(\mathbf{1}_{[j]_0}\circ \proj_{X_n}^2\mid \mathcal{B}_2^-\vee \mathcal{B}_0\big)\,  \d \eta_n + \sum_{j\in\mathcal A_{X_n}}\int_{S_3^{(a)}}  \Xi_n^j\, \d\eta_n\\
 &\geq \frac{ h_{\mu_1^n}(\sigma_{X_n}\mid\nu) +  h_{\mu_2^n}(\sigma_{X_n}\mid\nu)}{2} + \int_{S_3^{(a)}}  \Xi_n^{j_*} \,\d\eta_n.
\end{align*}
\end{proof}

Below we quote another result from \cite{Quas}. In \cite{Quas}, the statement is proved under the assumption that $X$ is a subshift of finite type, but the same proof works assuming that $X$ is a countable-state irreducible Markov shift and $Y$ is a subshift of finite type.

\begin{lemma}[{\cite[Lemma 3]{Quas}}]\label{lem:quas}
Consider a two-sided countable Markov shift \(X\subset \mathcal{A}_X^{\mathbb{Z}}\) with shift $\sigma_X$, a subshift $Y$ with shift $\sigma_Y$, and a \(1\)-block map \(\pi:X\to Y\) such that $\nu(\pi(X)) =1$ and $\supp(\nu) = Y$.
Fix an ergodic $\sigma_Y$-invariant probability $\nu$ on $Y$ and choose ergodic $\mu_1,\mu_2\in \pi^{-1}\{\nu\}$. Denote by $\widehat{\mu}:=\mu_1\otimes_{\nu}\mu_2$ their relatively independent joining on $X\times X$, and by $p_{X}^1,p_{X}^2:X\times X\to X$ the coordinate projections.

Put $\mathcal{F}^{-}:=\sigma(x_k;\,k<0)$ for the past \(\sigma\)-algebra on \(X\), and set
$\mathcal{F}^{-}_i:=(p^i_X)^{-1}(\mathcal{F}^{-})$, $i\in\{1,2\}$ on $X\times X$.  Write \(\mathcal Q_{Y}:=\sigma( [y_0]_0;\, y_0\in\mathcal A_Y)\) and define
\[
\mathcal{F}_0:=(\pi\circ p_{X}^1)^{-1}(\mathcal{Q}_{Y})=(\pi\circ p_{X}^2)^{-1}(\mathcal Q_{Y})\  (\mathrm{mod}\ \widehat{\mu}).
\]

Assume that $\widetilde{S}:=\{(u,v)\in X\times X;\ u_{-1}=v_{-1}\}$ has positive \(\widehat{\mu}\)-measure, and 
\[
\mathbb{E}_{\widehat{\mu}}\bigl(\mathbbm{1}_{[j]_0}\circ \proj_{X}^1 \,\bigm|\, \mathcal{F}^{-}_1 \vee \mathcal{F}_0\bigr)
=
\mathbb{E}_{\widehat{\mu}}\bigl(\mathbbm{1}_{[j]_0}\circ \proj_{X}^2 \,\bigm|\, \mathcal{F}^{-}_2 \vee \mathcal{F}_0\bigr)
\ \widehat{\mu}\text{-a.e. on } \widetilde{S},
\]
for every \(j\in\mathcal{A}_X\), then \(\mu_1=\mu_2\). In particular, there exists $a,b\in\mathcal A_X$ such that
\[\widehat{\mu}\left\{({u},{v})\in X\times X\,\left|\begin{aligned}
  & u_{-1}=v_{-1} = a, \text{and }\\ &\mathbb{E}_{\widehat{\mu}}\bigl(\mathbbm{1}_{[b]_0}\circ \proj_{X}^1 \,\bigm|\, \mathcal{F}^{-}_1 \vee \mathcal{F}_0\bigr)
\neq
\mathbb{E}_{\widehat{\mu}}\bigl(\mathbbm{1}_{[b]_0}\circ \proj_{X}^2\bigm|\, \mathcal{F}^{-}_2 \vee \mathcal{F}_0\bigr)
  \end{aligned} \right.\right\}>0.\]

\end{lemma}

\begin{proof}[Proof of Theorem \ref{thm:dream}] 

Let $a,b\in\mathcal A_X$ be as given at the end of Lemma \ref{lem:quas}, and let $n\in\mathbb N$ be large enough such that $[a]_0,[b]_0\in \alpha_n$; in this way we have that $a,b\in\mathcal A_{X_n}$.
We assume the notations of Theorem \ref{thm:quasin} and Lemma \ref{lem:quas}. Suppose, for contradiction, that the theorem is false. Then there exist two ergodic measures of maximal $\nu$-relative entropy, $\mu^1$ and $\mu^2$, such that
$$
\mu^1\otimes_\nu \mu^2\left(\{(u,v)\in X\times X;\, u_{0}=v_{0}\}\right)>0.
$$

Let $\eta:= \mu^1\otimes_\nu \mu^2 \otimes \zeta$, recalling that $\zeta$ is the $\mathrm{Bernoulli}(\frac{1}{2},\frac{1}{2})$  probability measure on $\{0,1\}^{\mathbb Z}$. 
Define $\mu_3=(\mathrm{proj}^3_X)_*\eta$, where for \(k\in\mathbb Z\) we set
$$N_k(u,v):=\sup\{m<k;\,u_m=v_m=a\}\ \text{with}\ \sup\varnothing=-\infty.$$
Recalling the definition of $r_{-\infty}$ in \eqref{eq:rinfty} we define
\[
\bigl(\mathrm{proj}^3_X(u,v,r)\bigr)_k:=\begin{cases}
u_k,&\text{if } r_{N_k(u,v)}=0,\\
v_k,&\text{if } r_{N_k(u,v)}=1.
\end{cases}
\]
Observe that $\proj_X^3$ is the non-truncated analogue of the  $\proj_{X_n}^3$ function in Section~\ref{sec:construction}, obtained by replacing \(X_n\) with \(X\).
We will show that
$$
h_{\mu^3}(X) > h_{\mu^1}(\sigma_X)=h_{\mu^2}(\sigma_X) =  h_{\mathrm{top}}(\sigma_{X}\mid \nu).$$
which yields a contradiction. Note that we cannot directly apply the result of Lemma \ref{lem:quas} since $(X,\sigma_X)$ is not necessarily a subshift of finite type.

We define $\mu^i_n = (\proj_{X_n})_*\mu^i$ for $i\in\{1,2\}$ and we recall that $\eta_n = \mu_1^n\otimes_\nu \mu_2^n \otimes \zeta$. Let $\mu_n^3$ be the $\sigma_{X_n}$-invariant probability measure constructed in Section \ref{sec:construction}. Observe that by our choice of $n\in\mathbb N$ the symbol $a\in\mathcal A_X$ also lies in the alphabet of its truncation $\mathcal A_{X_n}$.

We claim that $(\proj_{X_n})_* \mu^3 = \mu_n^3,$
which is equivalent (from the definitions of $\mu^3$ and $\mu_n^3$) to showing that
\begin{align}
(\proj_{X_n})_*  (\proj_{X}^3)_* \left(\mu^1\otimes_\nu \mu^2 \otimes \zeta\right) &= (\proj_{X_n}^3)_*   \left[ (\proj_{X_n})_*\mu^1\otimes_\nu (\proj_{X_n})_*(\mu^2) \otimes \zeta\right]\\
&=(\proj_{X_n}^3)_*   \left( \mu^1_n\otimes_\nu \mu^2_n \otimes \zeta\right) \label{eq:Condition1}
\end{align}
To show \eqref{eq:Condition1}, observe from \eqref{eq:Nn} and the fact that $a\in\mathcal A_{X_n}$ we have that
\begin{align}
    N_k(u,v) &= \sup\left\{m<k; u_m=v_m =a\right\} \nonumber= \sup\left\{m<k; \proj_{X_n}(u)_m=\proj_{X_n}(v)_m =a\right\}\\
    &= N_k^n(\proj_{X_n}(u),\proj_{X_n}(v)) \label{eq:Nk}
\end{align}
From \eqref{eq:Nk}, we obtain that for every $k\in\mathbb Z$
\begin{align*}\left(\proj_{X_n} \circ \proj_X^3(u,v ,r)\right)_k &=  
  \begin{cases}
    \proj_{X_n}(u)_k, & \text{if } r_{N_k(u,v)} = 0,\\[2pt]
     \proj_{X_n}(v)_k, & \text{if } r_{N_k(u,v)} = 1.
  \end{cases}\\
  &= \begin{cases}
    \proj_{X_n}(u)_k, & \text{if } r_{N_k^n(\proj_{X_n}(u),\proj_{X_n}(v))} = 0,\\[2pt]
     \proj_{X_n}(v)_k, & \text{if } r_{N_k^n(\proj_{X_n}(u),\proj_{X_n}(v))} = 1.
  \end{cases}\\
  &= \proj_{X_n}^3 (\proj_{X_n}(u), \proj_{X_n}(v), r)
  \end{align*}
Which implies \eqref{eq:Condition1} and therefore that $\proj_{X_n}\mu^3 = \mu_n^3$

In this way, we have that
$$
h_{\mu^i}(X)=\lim_{n\to\infty} h_{\mu^i_n}(X_n)\ \text{for }i\in\{1,2,3\}, \ \text{where}\ \mu_n^i := (\proj_{X_n})_*\mu^i.$$ 
By Lemma~\ref{lem:quas} there exists $b\in\mathcal A_X$ (which also lies in $\mathcal A_{X_n}$ for $n$ large enough) such that, on a $\mu^1\otimes_\nu \mu^2$-positive subset of $\{(u,v)\in X\times X:u_{-1}=v_{-1}=a\}$,
$$
\mathbb{E}_{\mu^1\otimes_\nu \mu^2}\bigl(\mathbbm{1}_{[b]_0}\circ p_{X}^1 \,\bigm|\, \mathcal{F}^{-}_1 \vee \mathcal{B}_0\bigr)
\neq
\mathbb{E}_{\mu^1\otimes_\nu \mu^2}\bigl(\mathbbm{1}_{[b]_0}\circ p_{X}^2 \,\bigm|\, \mathcal{F}^{-}_2 \vee \mathcal{B}_0\bigr).
$$
Hence
\begin{equation}\label{eq:ineq1}
\mathbb{E}_{\eta}\bigl(\mathbbm{1}_{[b]_0}\circ \proj_{X}^1 \,\bigm|\, \mathcal{B}^{-}_1 \vee \mathcal{B}_0\bigr)
\neq
\mathbb{E}_{\eta}\bigl(\mathbbm{1}_{[b]_0}\circ \proj_{X}^2 \,\bigm|\, \mathcal{B}^{-}_2 \vee \mathcal{B}_0\bigr)
\end{equation}
on an $\eta$-positive measure set contained in $$S^{(a)}:=\{\,z\in Z;\ (\proj_{X}^1 (z))_{-1}=(\proj_{X}^2 (z))_{-1}=a\}.$$

Let $\mathrm{proj}_{Z_n}:Z\to Z_n$ be the projection
\begin{align*}
    \mathrm{proj}_{Z_n}:Z&\to Z_n\\
    z =(x,y,r)&\mapsto (\proj_{X_n}(x),\proj_{X_n}(y),r),
\end{align*} 
and observe that $$\proj_{X_n}^i\circ\mathrm{proj}_{Z_n}=\proj_{X_n}\circ \proj_{X}^i\ \text{for any }i\in \{1,2,3\}.$$ For each $i\in\{1,2\}$ the sequence of $\sigma$-algebras
$$
(\mathrm{proj}_{Z_n})^{-1}\bigl(\mathcal B_{i,n}^{-}\vee \mathcal B_{0,n}\bigr),\qquad n\ge1,
$$
is increasing and
$$
\sigma\left(\bigvee_{n\ge1}(\mathrm{proj}_{Z_n})^{-1}\bigl(\mathcal B_{i,n}^{-}\vee \mathcal B_{0,n}\bigr)\right)=\mathcal B_i^{-}\vee \mathcal B_0.
$$
By Doob’s martingale convergence theorem,
\begin{equation}\label{eq:eq2}
\mathbb{E}_{\eta}\bigl(\mathbbm{1}_{[b]_0}\circ  \proj_{X}^i \,\bigm|\, \mathcal{B}^{-}_i \vee \mathcal{B}_0\bigr)
=
\lim_{n\to\infty}
\mathbb{E}_{\eta}\left(\mathbbm{1}_{[b]_0}\circ \proj_{X}^i \,\bigm|\,
(\mathrm{proj}_{Z_n})^{-1}(\mathcal B_{i,n}^{-}\vee \mathcal B_{0,n})\right)
\end{equation}
in $L^1(Z,\eta)$ and $\eta$-a.s. Moreover, $$\mathrm{proj}_{Z_n}(S^{(a)})\subset S_n^{(a)}:=\{\,z\in Z_n;\ (\proj_{X_n}^1 (z))_{-1}=(\proj_{X_n}^2 (z))_{-1}=a\,\}.$$

With $\psi(t):=-t\log t$,
\begin{align}
    &\int_{S_n^{(a)}}
\psi\left(\frac{A_n^{b}+B_n^{b}}{2}\right)
-\frac{1}{2}\psi(A_n^{b})-\frac{1}{2}\psi(B_n^{b})\, 
\mathrm d\eta_n\nonumber
\\&=
\int_{{\proj^{-1}_{Z_n}}(S_n^{(a)})}
\psi\left(\frac{A_n^{b}\circ \mathrm{proj}_{Z_n}+B_n^{b}\circ \mathrm{proj}_{Z_n}}{2}\right)
-\frac{1}{2}\psi(A_n^{b}\circ \mathrm{proj}_{Z_n})
-\frac{1}{2}\psi(B_n^{b}\circ \mathrm{proj}_{Z_n})\, 
\mathrm d\eta\nonumber
\\&\geq
\int_{S^{(a)}}
\psi\left(\frac{A_n^{b}\circ \mathrm{proj}_{Z_n}+B_n^{b}\circ \mathrm{proj}_{Z_n}}{2}\right)
-\frac{1}{2}\psi(A_n^{b}\circ \mathrm{proj}_{Z_n})
-\frac{1}{2}\psi(B_n^{b}\circ \mathrm{proj}_{Z_n})\, 
\mathrm d\eta,\label{eq:ineq4}
\end{align}
where $A_n^b,B_n^b$ are as in Theorem~\ref{thm:quasin} taking $j_*=b$. Observe that
\begin{align}
    A_n^{b}\circ \mathrm{proj}_{Z_n} &= \mathbb{E}_{\eta_n}\big(\mathbbm{1}_{[b]_0}\circ \proj_{X_n}^1 \,\big|\, \mathcal{B}_{1,n}^- \vee \mathcal{B}_{0,n}\big)\circ \mathrm{proj}_{Z_n} \nonumber \\
    &= \mathbb E_\eta \left(\mathbbm{1}_{[b]_0}\circ \proj_{X}^1 \,\big|\, (\mathrm{proj}_{Z_n})^{-1}\left(\mathcal{B}_{1,n}^- \vee \mathcal{B}_{0,n}\right)\right) \label{eq:3} 
\end{align}
Similarly, we obtain that 
  \begin{align}
      B_n^{b}\circ \mathrm{proj}_{Z_n} =\mathbb E_\eta \left(\mathbbm{1}_{[b]_0}\circ \proj_{X}^2 \,\big|\, (\mathrm{proj}_{Z_n})^{-1}\left(\mathcal{B}_{2,n}^- \vee \mathcal{B}_{0,n}\right)\right).\label{eq:ineq3}
  \end{align}
Combining \eqref{eq:ineq4} with \eqref{eq:3}, \eqref{eq:ineq3}, \eqref{eq:ineq1}, and \eqref{eq:eq2} yields
\begin{equation}\label{eq:key}
\lim_{n\to\infty}
\int_{S_n^{(a)}}\left[
\psi\left(\frac{A_n^{b}+B_n^{b}}{2}\right)
-\frac{1}{2}\psi(A_n^{b})-\frac{1}{2}\psi(B_n^{b})
\right]\mathrm d\eta_n >0.
\end{equation}

Finally, Theorem~\ref{thm:quasin} together with \eqref{eq:key} gives
$$
\begin{aligned}
h_{\mu^3}(\sigma_X\mid \nu)
&=\lim_{n\to\infty} h_{\mu^3_n}(\sigma_{X_n}\mid \nu)\\
&\geq \lim_{n\to\infty}\left[
\frac{h_{\mu^1_n}(\sigma_{X_n}\mid \nu)+h_{\mu^2_n}(\sigma_{X_n})}{2}
+\int_{S_n^{(a)}}\Xi_n(z)\,\eta_n(\mathrm dz)
\right]\\
&\geq \lim_{n\to\infty}\left[
\frac{h_{\mu^1_n}(\sigma_{X_n})+h_{\mu^2_n}(\sigma_{X_n})}{2}
+\int_{S_n^{(a)}}
\psi\left(\frac{A_n^{b}+B_n^{b}}{2}\right)
-\frac{1}{2}\psi(A_n^{b})-\frac{1}{2}\psi(B_n^{b})
\mathrm d\eta_n
\right]\\
&> \frac{h_{\mu^1}(\sigma_X\mid \nu)+h_{\mu^2}(\sigma_X\mid \nu)}{2}=h_{\mu^1}(\sigma_X\mid \nu).
\end{aligned}
$$
This contradicts  $h_{\mu^1}(\sigma_X\mid \mathbb \nu)= h_{\mathrm{top}}(\sigma_X\mid \nu)$, since $\pi_*\mu^3=\nu$.

\end{proof}

\section{Proofs of Theorem~A, Corollary~B and Proposition \ref{prop:example}}
\label{sec:5}

Before proving Theorem~\ref{thm:A}, we recall a result from the literature and prove a lemma. The proof of the Lemma below was sketched in Step~5 of the proof of Theorem~\ref{thm:1block}.

\begin{lemma}[{\cite[Proposition 13.2 and Step~2 of its proof]{Sarig2013}}] \label{lem:lift} Let $\Theta$ be as in Theorem \ref{thm:A} and fix $\chi>0$. Consider $\Sigma_s$ and $\pi_s:\Sigma_s\to \Omega\times M$ as in Theorem \ref{thm:1block}. Then, for each $\chi$-hyperbolic $\mu\in \mathcal M_e(\Theta)$
\begin{enumerate}
    \item[(i)]  there exists  $\widehat{\mu}\in M_e(\sigma_s)$  such that $(\pi_s)_*\widehat{\mu} = \mu$;
    \item[(ii)] $h_{\widehat{\mu}}(\sigma_s) = h_{\mu}(\Theta);$ and
    \item[(iii)] there exists $N_\mu$ such that $\# \pi_{s}^{-1}(\omega,x) = N_{\mu}$ for $\mu$-almost every $(\omega,x)\in \Omega\times M$.
\end{enumerate}
\end{lemma}

\begin{lemma}
\label{lem:count}Let $\Theta$ and $\mathbb P$ be as in Theorem \ref{thm:A}. Define 
$$
\mathcal M_{\mathrm{max}}(\Theta\mid \mathbb P) := \left\{
  \mu\in \mathcal M_e(\Theta)\middle|\;
  \begin{aligned}
  &  (\proj_\Omega)_*\mu = \mathbb P\ \text{and }\mu\ \text{is a hyperbolic}\\
  &\text{measure of }\mathbb P\text{-maximal relative entropy}
  \end{aligned}\ 
\right\},$$
If, for each $\chi>0$ and every homoclinic class $\mathcal{X}$, the set
\begin{align}
\mathcal M_{\mathrm{max}}^{\mathcal{X},\chi}(\Theta\mid \mathbb P) := \left\{
  \mu\in \mathcal M_e(\Theta)\middle|\;
  \begin{aligned}
  &  (\proj_\Omega)_*\mu = \mathbb P,\ \mu\ \text{is $\chi$-hyperbolic, $\mu(\mathcal{X})=1$, and }\\
  &\mu\ \text{is a hyperbolic measure of }\mathbb P\text{-maximal relative entropy}
  \end{aligned}\ 
\right\},\label{eq:Mchi}
\end{align}
is countable, then $\mathcal M_{\mathrm{max}}(\Theta\mid \mathbb P)$ is countable.
\end{lemma}
\begin{proof}
Define $\mathcal M_{\mathrm{max}}^{\chi}(\Theta\mid \mathbb P)$ as in \eqref{eq:Mchi} but changing $\mathcal{X}$ to $\Omega\times M$. Observe that
\begin{align*}
   \mathcal M_{\mathrm{max}}(\Theta\mid \mathbb P) = \bigcup_{\chi>0} \mathcal M_{\mathrm{max}}^{\chi}(\Theta\mid \mathbb P)=  \bigcup_{n\in\mathbb N} \mathcal M_{\mathrm{max}}^{\frac{1}{n}}(\Theta\mid \mathbb P).
\end{align*}
Thus it suffices to show that each $\mathcal M_{\mathrm{max}}^{\chi}(\Theta\mid \mathbb P)$ is countable.

Fix $\chi>0$. Let $\mu\in \mathcal M_e(\Theta)$ be $\chi$-hyperbolic. By Theorem~\ref{thm:katok},
$
\mu\big(\mathrm{NUH}'(\Theta)\big)=1.$ From Proposition~\ref{prop:homoclinicclasses}, there exists a countable family of Borel homoclinic classes
$\{\mathcal{X}_i\}_{i\in\mathcal I}$ forming a pairwise disjoint, $\Theta$-invariant partition of $\mathrm{NUH}'(\Theta)$. Since $\mu$ is ergodic, there is a unique $i_0\in\mathcal I$ such that $
\mu(\mathcal{X}_{i_0})=1.$ Consequently,
$$
\mathcal M_{\mathrm{max}}^{\chi}(\Theta\mid \mathbb P)
= \bigcup\left\{\mathcal M_{\mathrm{max}}^{\mathcal{X},\chi}(\Theta\mid \mathbb P)\ ;\ \mathcal{X} \text{ is a Borel homoclinic class}\right\}
= \bigcup_{i\in\mathcal I}\mathcal M_{\mathrm{max}}^{\mathcal{X}_i,\chi}(\Theta\mid \mathbb P).
$$
By hypothesis, each $\mathcal M_{\mathrm{max}}^{\mathcal{X}_i,\chi}(\Theta\mid \mathbb P)$ is countable, and $\mathcal I$ is countable; hence $\mathcal M_{\mathrm{max}}^{\chi}(\Theta\mid \mathbb P)$ is countable. Taking the countable union over $n\geq 1$ gives that $\mathcal M_{\mathrm{max}}(\Theta\mid \mathbb P)$ is countable.

\end{proof}

Now we can prove Theorem $\ref{thm:A}$.

\begin{proof}[Proof of Theorem \ref{thm:A}]
    Let $\mathbb{P}$ be a $\theta$-invariant ergodic probability measure with full support. By Lemma \ref{lem:count}, it suffices to show that $\mathcal{M}_{\max}^{\mathcal{X},\chi}(\Theta \mid \mathbb{P})$ is countable for each homoclinic class $\mathcal{X}$ and each $\chi>0$.

Fix $\chi>0$ and a Borel homoclinic class $\mathcal{X}$. Let $\Sigma_s$ be an irreducible locally compact topological Markov shift, $\Sigma_b$ a subshift of finite type, and let $\pi_s:\Sigma_s\to \mathcal{X}$ and $\pi_b:\Sigma_b\to \Omega$ be the Hölder maps given by Theorem \ref{thm:1block}. Recall that:
\begin{enumerate}
    \item $\pi_s\circ \sigma_s=\Theta\circ \pi_s$ and $\pi_b\circ \sigma_b=\theta\circ \pi_b$, where $\sigma_s$ and $\sigma_b$ denote the left shift maps on $\Sigma_s$ and $\Sigma_b$, respectively;
    \item $\pi_s[\Sigma_s^{\#}]$ has full measure for every $\chi$-hyperbolic $\mu\in \mathcal{M}_e(\Theta)$ with $\mu(\mathcal{X})=1$; moreover $\pi_s:\Sigma_s^{\#}\to \mathcal{X}$ is finite-to-one, and $\pi_b:\Sigma_b\to \Omega$ is a surjective finite-to-one map;
    \item there exists a $1$-block map $\mathrm{proj}_{\Sigma_b}:\Sigma_s\to \Sigma_b$ such that
    $$
    \pi_b\circ \mathrm{proj}_{\Sigma_b}=\mathrm{proj}_{\Omega}\circ \pi_s,
    $$
    with $\widehat{\mathbb P}[\mathrm{proj}_{\Sigma_b}(\Sigma_s)]=1$, where  $\widehat{\mathbb P}$ unique measure in $M_e(\sigma_b)$ such that $(\pi_b)_*\widehat{\mathbb P}=\mathbb P.$
\end{enumerate}

Let $\mu\in \mathcal{M}_e(\Theta)$ be a hyperbolic measure of $\mathbb{P}$-maximal relative entropy. By Lemma \ref{lem:lift}, there exists $\widehat{\mu}\in \mathcal{M}_e(\sigma_s)$ such that $(\pi_s)_*\widehat{\mu}=\mu$ and $ 
h_{\mu}(\Theta)=h_{\widehat{\mu}}(\sigma_s).$ Observe that
$$
(\pi_b)_*\big[(\mathrm{proj}_{\Sigma_b})_*\widehat{\mu}\big]
=(\pi_b\circ \mathrm{proj}_{\Sigma_b})_*\widehat{\mu}
=(\mathrm{proj}_{\Omega})_*(\pi_s)_*\widehat{\mu}
=(\mathrm{proj}_{\Omega})_*\mu
=\mathbb{P}.
$$
Thus $(\mathrm{proj}_{\Sigma_b})_*\widehat{\mu}=\widehat{\mathbb{P}}$, since the lift of $\mathbb P$ via $\pi_b$ is unique (recall that $\mathbb P$ is assumed to be fully supported on $\Omega$).

\medskip
\noindent\textbf{Claim.} For the factor map $\proj_{\Sigma_b}:(\Sigma_s,\sigma_s)\to(\Sigma_b,\sigma_b)$,  $\widehat{\mu}$ is a measure of maximal $\widehat{\mathbb{P}}$-relative entropy.  

\smallskip
\noindent\emph{Proof of the claim.}
Assume that the claim is false for a contradiction. Then, there exists $\widehat{\eta}\in \mathcal{M}_e(\sigma_s)$ with $(\mathrm{proj}_{\Sigma_b})_*\widehat{\eta}=\widehat{\mathbb{P}}$ such that $h_{\widehat{\eta}}(\sigma_s\mid \widehat{\mathbb{P}})>h_{\widehat{\mu}}(\sigma_s\mid \widehat{\mathbb{P}}).$ Set $\eta:=(\pi_s)_*\widehat{\eta}$. By item (1) above, $\eta$ is $\Theta$-invariant, and
$$
(\mathrm{proj}_{\Omega})_*\eta=(\mathrm{proj}_{\Omega}\circ \pi_s)_*\widehat{\eta}
=(\pi_b)_*(\mathrm{proj}_{\Sigma_b})_*\widehat{\eta}
=(\pi_b)_*\widehat{\mathbb{P}}=\mathbb{P}.
$$
By Lemma \ref{lem:lift}(iii) $\pi_s$ is $N_{\eta}$-to-one on a set of $\eta$-full measure. Moreover, recall that since $\mathbb P$ is fully supported on $\Omega$, $\pi_b$ is $\mathbb P$-a.s. one-to-one \cite[3.18]{BowenBook}. Using that finite extensions preserve entropy, we have
\begin{align*}
h_{\mu}(\Theta\mid \mathbb{P})
&=h_{\mu}(\Theta)-h_{\mathbb{P}}(\theta)
=h_{\widehat{\mu}}(\sigma_s)-h_{\widehat{\mathbb{P}}}(\sigma_b)
=h_{\widehat{\mu}}(\sigma_s\mid \widehat{\mathbb{P}})\\
&<h_{\widehat{\eta}}(\sigma_s\mid \widehat{\mathbb{P}})
=h_{\widehat{\eta}}(\sigma_s)-h_{\widehat{\mathbb{P}}}(\sigma_b)
=h_{\eta}(\Theta)-h_{\mathbb{P}}(\theta)
=h_{\eta}(\Theta\mid \mathbb{P}).
\end{align*}
Since $\eta(\mathcal{X})=\widehat{\eta}\big(\pi_s^{-1}(\mathcal{X})\big)=\widehat{\eta}(\Sigma_s)=1,
$ Proposition \ref{prop:homoclinicclasses} implies that $\eta$ is hyperbolic, which contradicts the maximality of $\mu$ for $\mathbb{P}$-relative entropy. Therefore $\widehat{\mu}$ is a measure of $\widehat{\mathbb{P}}$-maximal relative entropy. \qed

\medskip
For each $\mu\in \mathcal{M}_{\max}^{\mathcal{X},\chi}(\Theta\mid \mathbb{P})$ we have thus associated a measure $\widehat{\mu}\in \mathcal{M}_e(\sigma_s)$ such that $\widehat{\mu}$ is a measure of maximal $\widehat{\mathbb{P}}$-relative entropy. Moreover, $(\pi_b\circ \mathrm{proj}_{\Sigma_b})_*\widehat{\mu}=\mathbb{P}.$ Define
\begin{align*}
\mathcal{M}_{\max}(\Sigma_s\mid \mathbb{P})
&:=
\left\{
\widehat{\mu}\in \mathcal{M}_e(\sigma_s)\ ;\ 
\widehat{\mu}\ \text{has maximal }\widehat{\mathbb P}\text{-relative entropy}
\right\}
\end{align*}
Then, the assignment
\begin{equation}
\mu\in \mathcal{M}_{\max}^{\mathcal{X},\chi}(\Theta\mid \mathbb{P})
\mapsto
\widehat{\mu}\in \mathcal{M}_{\max}(\Sigma_s\mid \mathbb{P})
\label{eq:proj}
\end{equation}
is well defined and injective, since $(\pi_s)_*\widehat{\mu}=\mu$.

Hence, using the injectivity of \eqref{eq:proj}, we obtain that
\begin{align}
\#\, \mathcal{M}_{\max}^{\mathcal{X},\chi}&(\Theta\mid \mathbb{P})
\leq\# \left(
\left\{\widehat{\mu}\in \mathcal{M}_e(\sigma_s);\
\widehat{\mu}\ \text{has maximal $\widehat{\mathbb{P}}$-relative entropy}
\right\}\right).\label{eq:count1}
\end{align}

Observe that since $\mathbb P$ is fully supported on $\Omega$, then $\widehat{\mathbb P}$ is fully supported on $\Sigma_b$ (as a consequence of \cite[Theorem 3.19 and Proposition 3.19]{BowenBook}). From Corollary \ref{cor:dream}, with $X=\Sigma_s$, $Y=\Sigma_b=\mathrm{proj}_{\Sigma_b}(\Sigma_s)$, $\pi=\mathrm{proj}_{\Sigma_b}$ and $\nu=\widehat{\mathbb{P}}$ it follows that the set
$$\left\{\widehat{\mu}\in \mathcal{M}_e(\sigma_s):\, 
\widehat{\mu}\ \text{has maximal $\widehat{\mathbb{P}}$-relative entropy}
\right\}$$
is at most countable. Therefore from \eqref{eq:count1}  we obtain that the set $\mathcal{M}_{\max}^{\mathcal{X},\chi}(\Theta\mid \mathbb{P})$ is countable, which completes the proof.
\end{proof}

To prove Corollary \ref{cor:B}, we recall the following theorem for skew products, which refines Ruelle's inequality. We state a convenient weak version of the result presented in \cite{KatokKifer} that suffices for our purposes.

\begin{theorem}[{\cite[Chapter 5, Theorem 3.1.1]{KatokKifer}}]\label{thm:ruelle}
Let
$$\Theta:(\omega,x)\in\Omega\times M \mapsto \left(\theta(\omega),T_\omega(x)\right)\in \Omega\times M$$
be a $\mathcal C^{1+\alpha}$ skew-product diffeomorphism, and let $\mu\in\mathcal M_e(\Theta)$. Denote the base marginal by $\mathbb P := (\proj_{\Omega})_*\mu$. Let $\sigma(\mu)=\{\lambda_1,\ldots,\lambda_n\}$ be the Lyapunov spectrum of $\Theta$ (see Definition \ref{def:LE}). Define the set of \emph{fibre} Lyapunov exponents
\[
\sigma(\mu\mid \mathbb P)
:= \Bigl\{\lambda\in\sigma(\mu);\ 
\proj_{T_xM}\bigl(E^{\lambda}_{(\omega,x)}\bigr)\neq\{0\}\ \text{for }\mu\text{-a.e. }(\omega,x)\in \Omega\times M\Bigr\},
\]
and for each $\lambda\in\sigma(\mu\mid\mathbb P)$ set the (fibre) multiplicity
\[
m_\lambda = \dim\bigl(\proj_{T_xM} E^{\lambda}_{(\omega,x)}\bigr),
\]
which is $\mu$-a.s. constant by ergodicity.

Then
$$h_\mu(\Theta \mid \mathbb P) \le\sum_{\substack{\lambda \in \sigma(\mu\mid \mathbb P)\\ \lambda>0}} \lambda\, m_\lambda.$$
\end{theorem}

We can now prove Corollary \ref{cor:B}.

\begin{proof}[Proof of Corollary \ref{cor:B}]
 By virtue of Theorem \ref{thm:A}, it is sufficient to prove that every ergodic measure of $\mathbb P$-relative  maximal entropy of $\Theta$ is hyperbolic. In the notations of Theorem \ref{thm:ruelle}, it is clear that
$$\sigma(\mu) = \sigma(\theta)\cup \sigma(\mu\mid \mathbb P).$$ 

Since $\theta:\Omega\to \Omega$ is Axiom A then $0\notin \sigma(\theta)$. Using that $\mathrm{dim}(M) =2$ we have the cardinality of  $\sigma(\mu\mid \mathbb P)$ is either equal to $1$ or $2$. Therefore, to prove the corollary, it is enough to show that 
$$\sigma(\mu\mid \mathbb P) = \{\lambda_-,\lambda_+\}$$
with $\lambda_- < 0<\lambda_+$. From Theorem \ref{thm:ruelle} we have that 
\begin{align}
    0 < h_\mu(\Theta\mid \mathbb P) \leq \sum_{\substack{\lambda \in \sigma(\mu\mid \mathbb P)\\ \lambda>0}} \lambda\, m_\lambda. \label{eq:fwdineq}
\end{align}
On the other hand, applying the same result to the skew product $\Theta^{-1}$,  we obtain that 
\begin{align}
    0 <h_\mu(\Theta\mid \mathbb P) =  h_\mu(\Theta^{-1}\mid \mathbb P) \leq \sum_{\substack{\widetilde{\lambda} \in \sigma^{-1}(\mu\mid \mathbb P)\\ \widetilde{\lambda}>0}} \widetilde{\lambda} \, m_{\widetilde{\lambda}},\label{eq:reverseineq}
\end{align}
where $\sigma^{-1}(\mu\mid \mathbb P)$ denotes the set of the fibre Lyapunov exponents of $\Theta^{-1}$. Since the fibrewise Lyapunov spectrum of $\Theta^{-1}$  is the sign-reversal of that of $\Theta$, i.e. $\sigma^{-1}(\mu\mid \mathbb P)= \left\{-\lambda;\ \lambda\in \sigma(\mu\mid \mathbb P)\right\}.$
 It follows from \eqref{eq:reverseineq} that
\begin{align}
    0 <h_\mu(\Theta\mid \mathbb P) \leq \sum_{\substack{\widetilde{\lambda} \in \sigma^{-1}(\mu\mid \mathbb P)\\ \widetilde{\lambda}>0}} \widetilde{\lambda} m_{\widetilde{\lambda}} =  \sum_{\substack{\lambda \in \sigma(\mu\mid \mathbb P)\\ \lambda<0}} -\lambda m_{\lambda}. \label{eq:revineq2} 
\end{align}
From \eqref{eq:fwdineq} and \eqref{eq:revineq2} combined with $\mathrm{dim}(M)=2$, we obtain that ${\sigma(\mu\mid \mathbb P) = \{\lambda_-,\lambda_+\},}$ which implies that $\mu$ is a hyperbolic measure and concludes the proof of the theorem.

\end{proof}

We conclude this section by proving Proposition \ref{prop:example}.

\begin{proof}[Proof of Proposition \ref{prop:example}]

Consider $G_k:=\{x_1\in\mathbb S^1;\,|2-2\pi k\sin(2\pi x_1)|\geq 9\}.$ Choose $k_0$ large enough so that for every fixed $k\geq k_0$ we can pick disjoint intervals such that
$$\mathrm{length}(J_1)=\mathrm{length}(J_2)=\frac{1}{3},\ \text{and } J_1,J_2\subset G_k.
$$
Define the strips $S_1:=J_1\times\mathbb S^1$ and $S_2:=J_2\times\mathbb S^1\subset\mathbb T^2$;  see Figure \ref{fig:placeholder}. 
By construction,
$$
\{\,2x_1+x_2+k\cos(2\pi x_1)+f(\omega);\ x_1\in J_i\,\}=\mathbb S^1, \text{for }i\in\{1,2\},\ \text{every }x_2\in\mathbb S^1 \text{ and } \omega \in \mathbb T^2,
$$
and the map $x_1\in J_i\mapsto 2x_1-x_2+k\cos(2\pi x_1)+f(\omega)$ winds at least three times around $\mathbb S^1$ on each $i\in\{1,2\}$, since $|2-2\pi k\sin(2\pi x_1)|\geq 9$ and $\mathrm{length}(J_i)=1/3$.

\begin{figure}[hbt]
    \centering
\includegraphics[width=0.7\linewidth]{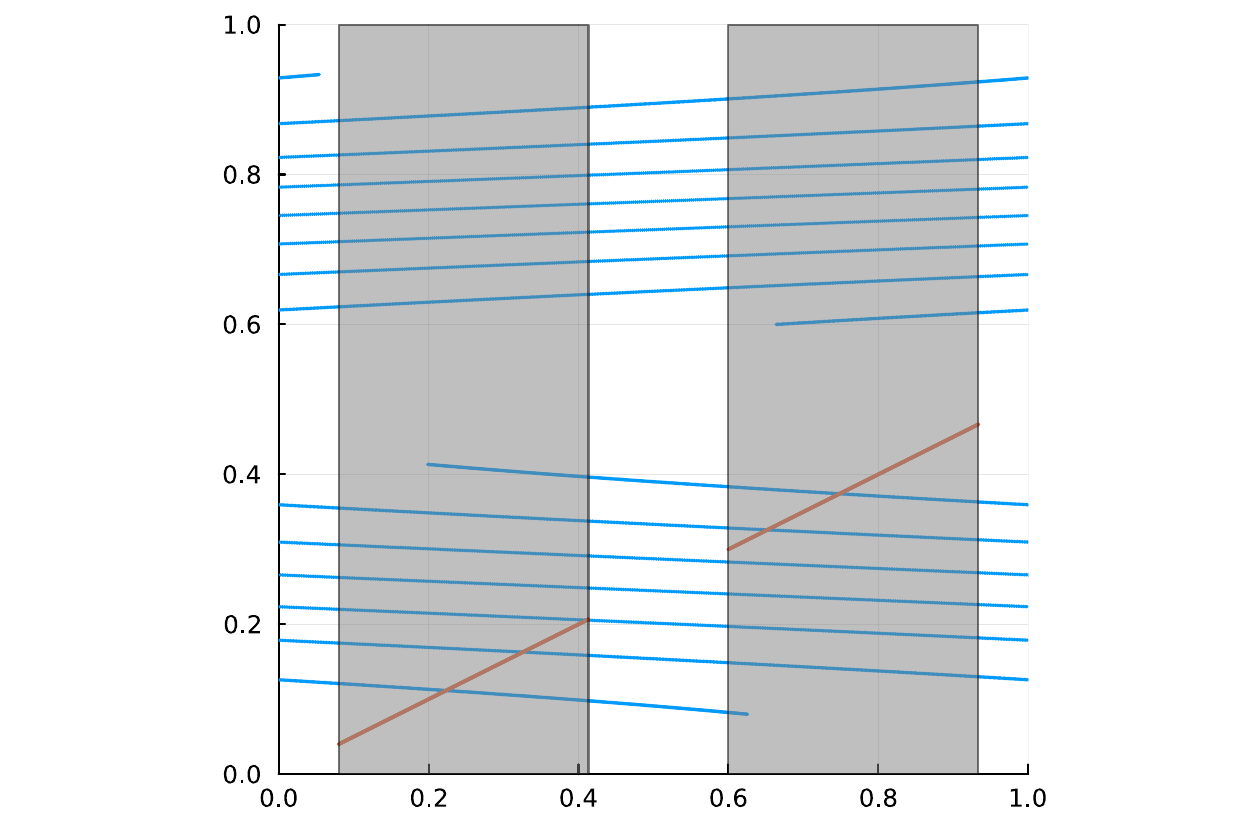}
    \caption{
Illustration of the graph transform in Steps~1–2 for $k=4$ and $\omega_{0}=0$.
The grey strips indicate $S_1$ (left) and $S_2$ (right).
The orange curves are the restrictions $\gamma_{1}=\gamma|_{J_{1}}\in\mathcal G_1$ and $\gamma_{2}=\gamma|_{J_{2}}\in\mathcal G_2$ of \(\gamma(t)=(t,t/2)\), while the blue curves are their images
$T_{\omega_{0}}\circ\gamma_{1}$ and $T_{\omega_{0}}\circ\gamma_{2}$.}
\label{fig:placeholder}
\end{figure}
\begin{step}[1]
Define, for $i\in\{1,2\}$, the family of graphs
$$
\mathcal G_i:=\big\{\ \gamma:J_i\to\mathbb T^2\ \text{ is }\mathcal C^1;\ \gamma(t)=(t,g(t))\ \text{with}\ \|g'\|_\infty<1\ \big\}.
$$
Then, for $k$ sufficiently large, for every $\omega\in\Omega$, every $\gamma\in\mathcal G_i$, and each $j\in\{1,2\}$, there exists $\widetilde{\gamma}= \widetilde{\gamma}(\omega,\gamma,j)\in\mathcal G_j$ such that
$$
\widetilde{\gamma}(J_j)\subset T_\omega\big(\gamma(J_i)\big).
$$
\end{step}

Indeed, for $\gamma(t)=(t,g(t))$ we compute
$$
T_\omega\circ\gamma(t)=\big(2t-g(t)+k\cos(2\pi t)+f(\omega),\ t\big).
$$
For $t\in J_i\subset G_k$,
$$
\left|\frac{d}{dt}\big(2t-g(t)+k\cos(2\pi t)+f(\omega)\big)\right|
\geq |2-2\pi k\sin(2\pi t)|-|g'(t)|\ \ge\ 9-1\ =\ 8.
$$
Hence $t\mapsto 2t-g(t)+k\cos(2\pi t)+f(\omega)$ is strictly monotone on $J_i$ and, since $\mathrm{length}(J_i)=1/3$, its image wraps at least twice around the horizontal circle. Let $h:\mathbb S^1\rightarrow J_i$
be any inverse branch of
$$
t\in J_1 \mapsto 2t-g(t)+k\cos(2\pi t)+f(\omega) \in \mathbb S^1
$$
Then
$$
\{(s,h(s));\ s\in J_j\}\subset \{\,T_\omega\circ\gamma(t);\ t\in J_i\,\}.
$$
Define $\widetilde{\gamma}(s):=(s,h(s))$. By the chain rule and the identity
$$
h\big(2t-g(t)+k\cos(2\pi t)+f(\omega)\big)=t,
$$
we obtain that
$$
\sup_{y\in J_j}|h'(y)|
=\sup_{t\in J_i}\left|\frac{1}{2-g'(t)-2\pi k\sin(2\pi t)}\right|
\leq\frac{1}{8}<1.
$$
Therefore, $\widetilde{\gamma}\in\mathcal G_j$ and Step 1 follows.

\begin{step}[2]
Given a sequence $\un r=(r_\ell)_{\ell\in\mathbb Z}\in\{1,2\}^{\mathbb Z}$ and any $\omega\in\Omega$, there exists $x\in\mathbb T^{2}$ such that $T_\omega^{\ell}(x)\in S_{r_\ell}$, for every $\ell\in\mathbb N.$
\end{step}

Recall the definition of the family $\mathcal G_\ell$ from Step~1, with $\ell\in\{1,2\}$. Consider the graph $\gamma_{0}\colon J_{r_0}\to\mathbb T^{2}$ defined by $\gamma_{0}(t)=(t,0)$. Clearly $\gamma_{0}\in\mathcal G_{r_0}$. By Step~1, there exists $\gamma_{1}\colon J_{r_1}\to\mathbb T^{2}$ with $\gamma_{1}\in\mathcal G_{r_1}$ such that
$$
\gamma_{1}(J_{r_1})\subset T_\omega\big(\gamma_{0}(J_{r_0})\big).
$$
Proceeding inductively, we construct graphs $\{\gamma_{\ell}\colon J_{r_\ell}\to\mathbb T^{2}\}_{\ell\in\mathbb N}$ with $\gamma_\ell\in\mathcal G_{r_\ell}$ such that for each $\ell\geq 1$,
\begin{align}\gamma_{\ell}(J_{r_\ell})&\subset  T_{\theta^{\ell-1}\omega} \circ \gamma_{\ell-1}(J_{r_{\ell-1}})\nonumber\\
&\subset \ldots \subset T_{\theta^{\ell-1}\omega}\circ\cdots\circ T_{\theta\omega}\circ T_\omega\circ\gamma_{0}(J_{r_0})
= T_\omega^{\ell}\circ\gamma_{0}(J_{r_0}).\label{eq:inclusion}
\end{align}

For each $\ell\in\mathbb N$ define the compact sets
$$
\mathcal{T}_\ell:=\big\{\,t\in J_{r_0};\ T_\omega^{j}\big(\gamma_{0}(t)\big)\in S_{r_j}\ \text{ for every }j\in\{0,\ldots, \ell\}\,\big\}.
$$
By construction $\gamma_{\ell} \in \mathcal G_{r_\ell}$. Therefore, there exists $g_\ell:J_{r_\ell}\to \R$ such that $\gamma_\ell(t) = (t,g_\ell(t))$.
Thus, $\gamma_\ell(J_{r_\ell})=\{(t,g_\ell(t));t\in J_{r_\ell}\}\subset S_{r_\ell}$, and from \eqref{eq:inclusion} we obtain that $(\mathcal{T}_\ell)_{\ell\in\mathbb N}$ is a nested sequence of nonempty compact sets, hence
$$
\bigcap_{\ell=0}^{\infty}\mathcal{T}_\ell\neq\varnothing.
$$
Choose $t_0\in\bigcap_{\ell=0}^{\infty}\mathcal{T}_\ell$ and set $x:=\gamma_{0}(t_0)=(t_0,0)$. Then, for every $\ell\in\mathbb N$,
$$
T_\omega^{\ell}(x)=T_\omega^{\ell}\big(\gamma_{0}(t_0)\big)\in S_{r_\ell},
$$
which proves Step 2.

\begin{step}[3]
We show that $h_{\mathrm{top}}(\Theta \mid \mathbb P)>0$, which completes the proof of the proposition.
\end{step}
Given $\omega\in\Omega$, define
$$
X_\omega:=\bigcap_{n=0}^{\infty} T_\omega^{-n}\left(J_1\sqcup J_2\right).
$$
By Step 2, $X_\omega \neq \varnothing$, and the coding map
\begin{align*}
h_\omega\colon X_\omega &\rightarrow \{1,2\}^{\mathbb N} \\
x &\mapsto \big(J(T_\omega^{i}x)\big)_{i\geq 0},
\end{align*}
where $J(z)=i$ if $z\in J_i$ for $i\in\{1,2\}$, is surjective.

Set $X_\Omega:=\{(\omega,x);\ \omega\in\Omega\ \text{and}\ x\in X_\omega\}.$ Define the forward-time semi-conjugacy
$$
\pi\colon (\omega,x)\in X_\Omega \mapsto (\omega,h_\omega(x))\in \Omega\times\{1,2\}^{\mathbb N}.
$$
Let $\sigma\colon\{1,2\}^{\mathbb N}\to\{1,2\}^{\mathbb N}$ be the left shift and set
$$
\sigma_\Omega\colon (\omega,s)\in \Omega\times\{1,2\}^{\mathbb N}\to (A\omega,\sigma(s))\in \Omega\times\{1,2\}^{\mathbb N}.
$$
By construction, $\pi\circ \Theta=\sigma_\Omega\circ \pi.$
 Since $\pi$ is a surjective factor map and $\Theta$ and $\sigma_\Omega$ share the same base $(\Omega,A)$, we obtain
$$
h_{\mathrm{top}}(\Theta\mid \mathbb P) \geq h_{\mathrm{top}}(\sigma_\Omega\mid \mathbb P) =h_{\mathrm{top}}(\sigma) =\log 2>0.
$$

\end{proof}

\section*{Acknowledgements}

The authors acknowledge Yuri Lima for pointing out an inconsistency in the proof of Theorem \ref{thm:1block} in an earlier version, and Anthony Quas for valuable discussions and for suggesting an improved approach for the proof of Theorem \ref{thm:dream}.
GF and MMC are supported by an Australian Research Council Laureate Fellowship (FL230100088).

\printbibliography

\end{document}